\documentclass[a4paper,11pt,english,british, hidelinks]{article}
\usepackage{setspace,graphicx,
epstopdf,amsmath,amsfonts,amsgen, mathtools,
amstext,amsthm,amsbsy,amsopn,amssymb,
bbm,
parskip,verbatim,mathrsfs,enumerate,
xcolor,comment}  
\usepackage[utf8]{inputenc}   
\usepackage[top=2.2cm, bottom=2.8cm, left=2cm, right=2 cm]{geometry}
\usepackage[square,numbers]{natbib} 
\usepackage{todonotes}
\usepackage{authblk}
\usepackage[pdfborder={0 0 0}]{hyperref}
\usepackage{soul}
\usepackage{hyperref}
\hypersetup{hidelinks}

\hypersetup{pdfborder= {0 0 0},colorlinks=false, linkbordercolor={1 0 0}}

\let\oldsqrt\sqrt
\def\sqrt{\mathpalette\DHLhksqrt}
\def\DHLhksqrt#1#2{%
\setbox0=\hbox{$#1\oldsqrt{#2\,}$}\dimen0=\ht0
\advance\dimen0-0.2\ht0
\setbox2=\hbox{\vrule height\ht0 depth -\dimen0}%
{\box0\lower0.4pt\box2}}

\def\bn{\boldsymbol{n}}
\def\bm{\boldsymbol{m}}

\def\b0{\boldsymbol{0}}

\newcommand{\R}     {\mathbb{R}} 
 
\newcommand{\N}     {\mathbb{N}} 
\renewcommand{\P}   {\mathbb{P}} 
 
\newcommand{\E}     {\mathbb{E}}

\newcommand{\Ccal}   {{\mathcal C }} 
 
\newcommand{\Ecal}   {{\mathcal E }} 
 
\newcommand{\Gcal}   {{\mathcal G }}

\newcommand{\Kcal}   {{\mathcal K }} 
 
\newcommand{\Mcal}   {{\mathcal M }}

\newcommand{\Pcal}   {{\mathcal P }}

\newcommand{\Scal}   {{\mathcal S }} 
 
\newcommand{\Ucal}   {{\mathcal U }} 
\newcommand{\Vcal}   {{\mathcal V }} 
\newcommand{\Wcal}   {{\mathcal W }} 
\newcommand{\Xcal}   {{\mathcal X }}

\newcommand{\Exp}{\mathscr{E}\kern-0.2mm{\operatorname{xp}}}
\newcommand{\Log}{\mathscr{L}\kern-0.2mm{\operatorname{og}}}

\def\1{{\mathchoice {1\mskip-4mu\mathrm l}      
{1\mskip-4mu\mathrm l} 
{1\mskip-4.5mu\mathrm l} {1\mskip-5mu\mathrm l}}}

\usepackage{mathtools}

\numberwithin{equation}{section}
\numberwithin{figure}{section}
\newtheoremstyle{plain}
  {6pt}
  {4pt}
  {\slshape}
  {}
  {\bfseries}
  {.}
  {0.5em}
  {}%
\newtheorem{thm}{\protect\theoremname}
  \newtheorem{defn}[thm]{\protect\definitionname}
  
  \newtheorem{prop}[thm]{\protect\propositionname}

  \newtheorem{lem}[thm]{\protect\lemmaname}
  \numberwithin{thm}{section}

\usepackage[english]{babel}
  \addto\captionsbritish{\renewcommand{\corollaryname}{Corollary}}
  \addto\captionsbritish{\renewcommand{\definitionname}{Definition}}
  \addto\captionsbritish{\renewcommand{\factname}{Fact}}
  \addto\captionsbritish{\renewcommand{\propositionname}{Proposition}}
  \addto\captionsbritish{\renewcommand{\remarkname}{Remark}}
  \addto\captionsbritish{\renewcommand{\theoremname}{Theorem}}
  \addto\captionsenglish{\renewcommand{\corollaryname}{Corollary}}
  \addto\captionsenglish{\renewcommand{\definitionname}{Definition}}
  \addto\captionsenglish{\renewcommand{\factname}{Fact}}
  \addto\captionsenglish{\renewcommand{\propositionname}{Proposition}}
  \addto\captionsenglish{\renewcommand{\remarkname}{Remark}}
  \addto\captionsenglish{\renewcommand{\theoremname}{Theorem}}
  \providecommand{\corollaryname}{Corollary}
  \providecommand{\definitionname}{Definition}
  \providecommand{\factname}{Fact}
  \providecommand{\propositionname}{Proposition}
  \providecommand{\remarkname}{Remark}
\providecommand{\theoremname}{Theorem}
\providecommand{\lemmaname}{Lemma}

\title{Exponential decay of transverse correlations 
\\
for $O(N)$  spin systems and related models
 }
\author[1]{Benjamin Lees\footnote{Heilbronn Institute for Mathematical Research and School of Mathematics, University of Bristol. Email: \textit{benjamin.lees@bristol.ac.uk}}}
\author[2]{Lorenzo Taggi\footnote{Sapienza Universit\'a di Roma. Email: \textit{lorenzo.taggi@uniroma1.it}}}
\affil[1]{Heilbronn Institute - University of Bristol}
\affil[2]{Sapienza Universit\'a di Roma}

 \date{\today}

\begin{document}
\maketitle 

\begin{abstract}
We prove exponential decay of transverse correlations in the Spin $O(N)$ model for arbitrary (non-zero) values of the external magnetic field and arbitrary  spin dimension $N > 1$. Our result is new when $N > 3$, in which case no Lee-Yang theorem is available,
it  is an alternative to Lee-Yang when $N = 2, 3$, and also holds for a wide class of multi-component spin systems with continuous symmetry.
The  key ingredients  are a representation of the model as a system of coloured random paths,  a `colour-switch' lemma, and a 
 sampling procedure which allows us to  bound from above the `typical' length of the open paths.
\end{abstract}

\section{Introduction}
The Spin $O(N)$ model is a classical statistical mechanics model whose configurations are collections of unit vectors, called spins, taking values on the surface of a $N-1$ dimensional unit sphere, $\mathbb{S}^{N-1} \subset \mathbb{R}^N$, with each spin associated to the vertex of a graph. Some special cases of the Spin $O(N)$ model are the Ising model ($N=1$), the XY model ($N=2$),  and  the classical Heisenberg model ($N = 3$). Despite the fact that  it is a very classical model, there remain important gaps in understanding, particularly in the case $N>2$. This paper addresses a basic and important question, namely how fast do correlations between spins decay with the distance between their associated vertices when a non-zero external magnetic field is present? More concretely, we consider transverse correlations in the presence of an external magnetic field parallel to the $\boldsymbol{e}_N$ cartesian vector of arbitrary (non-zero) intensity, namely correlations between the $i$-th component of the spins for any $i \in \{1, \ldots, N - 1\}$. Our main result states that, for any value of the inverse temperature and any non-zero value of the external magnetic field, transverse correlations decay exponentially fast with the graph distance between the two vertices (in the literature one refers to the exponential decay of correlations as a \textit{mass-gap condition}). Our proof method is probabilistic, it uses a  new  representation of the model
as a system of random walks and loops, which employs \textit{colours} and \textit{pairings}, and a sampling procedure which allows us to stochastically bound the length of a random walks of a given colour by `exploring' the realisation `step by step', thus enabling a comparison with a simpler stochastic process.

When $N = 1, 2, 3$,  the mass-gap condition for arbitrary non-zero value of the external field is a  consequence of  the cluster expansion and of the celebrated Lee-Yang theorem (see the recent papers \cite{F-PFR, F-PFR2}, an alternative approach for the $N=1$ case is presented in \cite{Ott}).
The Lee-Yang theorem was proved in  \cite{L-S} when $N=2$ (in the same paper results involving the $N>2$ cases
are also derived, but these require anisotropic coupling constants),
and in \cite{Asano, Dunlop, S-F} when $N = 3$, by taking an appropriate limit of the corresponding quantum system.
In the absence of a Lee-Yang theorem when $N>3$, the cluster expansion provides only perturbative results, i.e, the mass-gap condition can only be proven for large  enough  (positive or negative) values of the external magnetic field.

Our result is new when $N > 3$ and, for any $N>1$, our method provides a new direct proof of the  mass-gap condition for transverse correlations bypassing the Lee-Yang analiticity result and the cluster expansion. 
Additionally, our proof is also quite flexible and, for example, it holds for any graph of bounded degree,  it holds on $\mathbb{Z}^d$ with finite range (not necessarily translation invariant) coupling constants, and it holds for a class of models with continuous symmetry whose interaction does not necessarily take the form $e^{ - \mathcal{H}}$ (with $\mathcal{H}$ representing the hamiltonian function) -- these models are `less physical' but they lead to interesting random loop models, for example the loop O(N) model \cite{Chayes,  D-Copin, S-P, T2} (see Section \ref{sect:extensions}).

\subsection{Model and main result}
\label{sec:definition}
We define the Spin $O(N)$ model on an arbitrary graph  with uniform coupling constants and zero boundary conditions
and we refer to Section \ref{sect:extensions} for extensions.
Consider a finite simple graph $\Gcal=(\Vcal,\Ecal)$ and, for $N\in\N_{>0}$, define the configuration space $\Omega_{\Gcal,N} :=(\mathbb S^{N-1})^{\Vcal}$,
where $\mathbb{S}^{N-1} \subset \mathbb{R}^N$ is the $N-1$ dimensional unit sphere. 
For $\beta\geq 0$ and $h\in \R$ we introduce the hamiltonian function acting on $\varphi=(\varphi_x)_{x\in \Vcal}\in\Omega_{\Gcal,N}$,
\begin{equation}\label{eq:hamiltonian}
H^{spin}_{\Gcal,N,\beta,h}(\varphi)= -  \beta \sum_{\{x,y\}\in \Ecal}\varphi_x\cdot\varphi_y -  h\sum_{x\in \Vcal}\varphi^N_x,
\end{equation}
where $\, \cdot \, $ denotes the usual inner product on $\R^N$, the first sum is over undirected edges,  and $\varphi^i_x$ is the $i^{th}$ component of the vector $\varphi_x\in\mathbb S^{N-1}\subset\R^N$. We define the expectation operator $\langle\cdot\rangle^{spin}_{\Gcal,N,\beta,h}$ acting on $f:\Omega_{\Gcal,N}\to \R$ by
\begin{equation}
\langle f\rangle^{spin}_{\Gcal,N,\beta,h}=\frac{1}{Z^{spin}_{\Gcal,N,\beta,h}}\int_{\Omega_{\Gcal,N}}\mathrm{d}\varphi \, f(\varphi) \, e^{-H^{spin}_{\Gcal,N,\beta,h}(\varphi)},
\end{equation}
where $\mathrm{d}\varphi = \prod_{x \in \mathcal{V}} \mathrm{d}\varphi_x$ is a product measure with $\mathrm{d}\varphi_x$ the uniform measure on $\mathbb{S}^{N-1}$ and $Z^{spin}_{\Gcal,N,\beta,h}$ is a normalising constant that ensures $\langle 1\rangle^{spin}_{\Gcal,N,\beta,h}=1$. Our main result concerns correlations between spins $\varphi_x$, $\varphi_y$ when the graph distance from $x$ to $y$, $d_{\Gcal}(x,y)$, is large.
For any $x \in \mathcal{V}$, define the random variable $S_x : \Omega_{\mathcal{G}, N} \mapsto \mathbb{S}^{N-1}$
representing the spin at $x$ as,
$
S_x(\varphi) : = \varphi_x,
$
moreover we represent its components as 
$
S_x = (S_x^1, \ldots, S_x^N ).
$

\hypertarget{c0}{
\begin{thm}\label{thm:maintheorem} 
Let $\mathcal{G}$ be an infinite simple graph with  bounded degree.
For any $h\neq 0$, $\beta \geq 0$ and $N\in \N_{\geq 2}$ there are positive constants $c_0 = c_0 (\mathcal{G}, \beta,h, N)$ and 
$C_0 = C_0 (\mathcal{G}, \beta,h, N)$
such that the following holds. Let  $(\mathcal{G}_L)_{L \in \mathbb{N}}$, with 
$\mathcal{G}_L = ( \mathcal{V}_L, \mathcal{E}_L)\subset \Gcal$,
be an arbitrary sequence of finite graphs. Then,
for any $L \in \mathbb{N} $ and any $x, y \in \mathcal{V}_L$,
\begin{equation}\label{eq:mainclaim}
 \langle S^1_x S^1_y\rangle^{spin}_{\mathcal{G}_L,N,\beta,h} \leq C_0 e^{-c_0 \, d_{\Gcal}(x,y)},
\end{equation}
where $d_{\Gcal}(x,y)$ denotes the graph distance between $x$ and $y$ in $\Gcal$. Moreover,  the choice of $c_0$ can be made so that, $c_0 = O(h^2)$ in the limit as $h \rightarrow 0$.
\end{thm}
}
For example, our theorem holds when $\mathcal{G} = \mathbb{Z}^d$ and $\Gcal_L$ is a box of side length $L$ or when $\mathcal{G}$ is a regular tree and $\Gcal_L$ is the subtree of depth $L$. Our  result also holds for 
 non-zero boundary conditions, on $\mathbb{Z}^d$ with finite range (not necessarily translation invariant) coupling constants,  and for spin systems whose measure is not necessarily in the form $e^{ - \mathcal{H}}$, see extensions in Section \ref{sect:extensions}.

\subsection{Proof method}
The first step of the proof is a  representation of the Spin $O(N)$ model as a system of random undirected walks and loops, which may overlap and intersect each other. We collectively refer to walks and loops as \emph{paths}. Each path is given a \textit{colour} $i \in \{1, \ldots, N\}$ and the measure involves an  on-site weight function that penalises large numbers of overlaps. This representation corresponds to a combination of the ones introduced in \cite{B-U, L-T}, which are in turn related to the one of Brydges, Fr\"ohlich and Spencer \cite{B-F-S}, and the random current representation of the Ising model \cite{Aiz}. 
In our representation a  \emph{ghost vertex}, denoted by $g$, is added to the graph, with edges to each other vertex representing the external field. The correlation between the first component of the spins at $x$ and $y$ can be written as a ratio of two partition functions, the one in the denominator refers to a gas of loops of any colour and walks of colour $N$ ($N$-walks) with both end points at the ghost vertex,
the one in the numerator has,  in addition, a $1$-walk with end points $x$ and $y$.

The first (simple, but important)  step of our analysis
is a `colour-switch lemma'. We use a map which `transforms' the partition function in the numerator by switching  the colour of the $1$-walk to $N$ and adding two more steps to the walk that connect its end-points to the ghost vertex. This transformation allows us to show that the spin correlation equals  the expected number of $N$-walks with their two last steps  on  the edges  $\{x,g\}$ and $\{y,g\}$.
 
By the colour-switch lemma, deriving the exponential decay of transverse correlations is equivalent to showing that the expected number of such $N$-walks is exponentially small with respect to $d_{\Gcal}(x,y)$. The general idea of the proof is that every walk which starts from the edge $\{x,g \}$ has a positive probability to be paired to the ghost vertex at each of its steps, thus `dying' at that step, hence it cannot be too long.

The two main mathematical ingredients  for turning such a simple description into a  rigorous proof  are: \textbf{(i)} An upper bound  on the distribution of the \textit{local times},  which is defined as the number of visits of walks or loops to a vertices. A small local time is required since we can show that the probability that a walk `dies' at a given vertex is uniformly bounded away from 0 if the local time at that vertex not too large.
\textbf{(ii)} A \textit{sampling procedure}, which consists of sampling the random path configuration step by step by exploiting the spatial Markov property, thus controlling the various (many) dependencies  by enabling the comparison with simpler stochastic processes.

\paragraph{Organisation.}
In Section \ref{sect:representation} we introduce the random path representation of the Spin $O(N)$ model in the presence of an external magnetic field and present the colour-switch lemma. In Section \ref{sec:linkbounds} we provide bounds for the  distribution of the local times. In Section \ref{sect:samplingprocedure} we introduce
the sampling procedure. In Section \ref{sect:proof of theorem}
we present the proof of our main theorem and discuss some extensions.

\subsubsection*{Notation}
\begin{center}
	\begin{tabular}{ l l }
	
$\mathcal{G} = ( \mathcal{V}, \mathcal{E})$ &  an undirected, simple, finite graph \\

$G=(V,E)$ & the graph $\Gcal$ together with the ghost vertex $g$ \\

$e \in {E}$ or $\{x,y\} \in {E}$ & undirected edges \\

$x \sim y$ & two neighbour vertices, i.e, $x, y \in V$ such that $\{x,y\} \in E$ \\
 $N \in \mathbb{N}_{>0}$& the number of colours\\
 $[N]$ & $\{1,\dots,N\}$\\

$d_x,d^*_{\Gcal}$ &  the graph degree of $x\in\Vcal$ and $\max_{x\in\Vcal}d_x$\\

$d_\Gcal(x,y)$ &  the graph distance between $x$ and $y$ \\

$\mathcal{M}_{\mathcal{G}}$ & the set of link cardinalities on $\mathcal{G}$ (with  $G$ possibly replacing $\Gcal$)   \\

$\mathcal{C}_{\mathcal{G}}(m)$ & the set of colourings for  $m\in\mathcal{M}_{\mathcal{G}}$   \\


$\mathcal{P}_{\mathcal{G}}(m, c)$ & the set of pairing configurations for $m\in\mathcal{M}_{\mathcal{G}}$ and $c\in\mathcal{C}_{\mathcal{G}}(m)$   \\
 
$w= (m,c, \pi)$ &  a wire configuration with $m \in\mathcal{M}_{\mathcal{G}}$, $c \in\mathcal{C}_{\mathcal{G}}(m)$, and
$\pi \in \mathcal{P}_{\mathcal{G}}(m, c)$  \\ 

$\mathcal{W}_\Gcal$ &  the set of wire configurations on $\mathcal{G}$ \\

$n^i_x(w)$ & the local time of $i$-objects at $x$ \\

$n_x(w)$ & $ \sum_{i=1}^{N} n_x^i(w)$\\

$u^i_x(w)$ & the number of unpaired $i$-links at $x$\\

$v^i_x(w)$ & the number of pairings of $i$-links at $x$\\






$Z_{\Gcal,N,\beta,U}(x,y)$ 
& the total weight of configurations with a 1-path from $x$ to $y$\\


$\mathbb{G}_{\Gcal,N,\beta,U}(x,y)$  
& {the two-point function between $x$ and $y$ in the random path model} \\



 
\end{tabular}
\end{center}

\section{The random path representation}
\label{sect:representation}
In this section we introduce a random path representation for the Spin $O(N)$ model in the presence of an external magnetic field. We refer to this representation as the \emph{Random Path Model} (RPM).
This representation corresponds to a combination of the 
one introduced in  \cite{B-U, L-T}, 
which was also used in \cite{T} in the study of the dimer model in $\mathbb{Z}^d$, $d \geq 3$,
and of the random current representation of the Ising model \cite{Aiz}.
Two key aspects of the representation are  \textit{pairings} and \textit{colours}, these are two  ingredients which are not present (or necessary) in the $N=1$ case 
\cite{Aiz}, the well-known Ising model, but which play a crucial role  in our analysis, which involves the $N >1 $ cases.
A random loop model (of different nature than ours) was also used in  \cite{Bj-U} for the study of quantum spin systems.

\subsection{Random path model}
We consider a general finite undirected simple graph $\mathcal{G} = (\mathcal{V}, \mathcal{E})$.
Let $N \in \mathbb{N}_{>0}$ be the \textit{number of colours}. 
A realisation of the RPM can be viewed as a collection of undirected (closed or open) paths with colours in $[N]:=\{1,\dots,N\}$.
A path is identified by a collection of links, a colouring and by pairings. 

To begin, denote by  $m \in  \mathcal{M}_{\mathcal{G}} := \mathbb{N}^{\Ecal}$ a \textit{collection of links} on $\Ecal$. More specifically, 
$$m = \big ( m_e \big )_{e \in \mathcal{E}},$$
where $m_{e}\in\N$ represents the number of links on $e\in \mathcal{E}$. We say a link is incident to $x\in \Vcal$ if it is on an edge incident to $x$.
\begin{figure}
\includegraphics[scale=0.40]{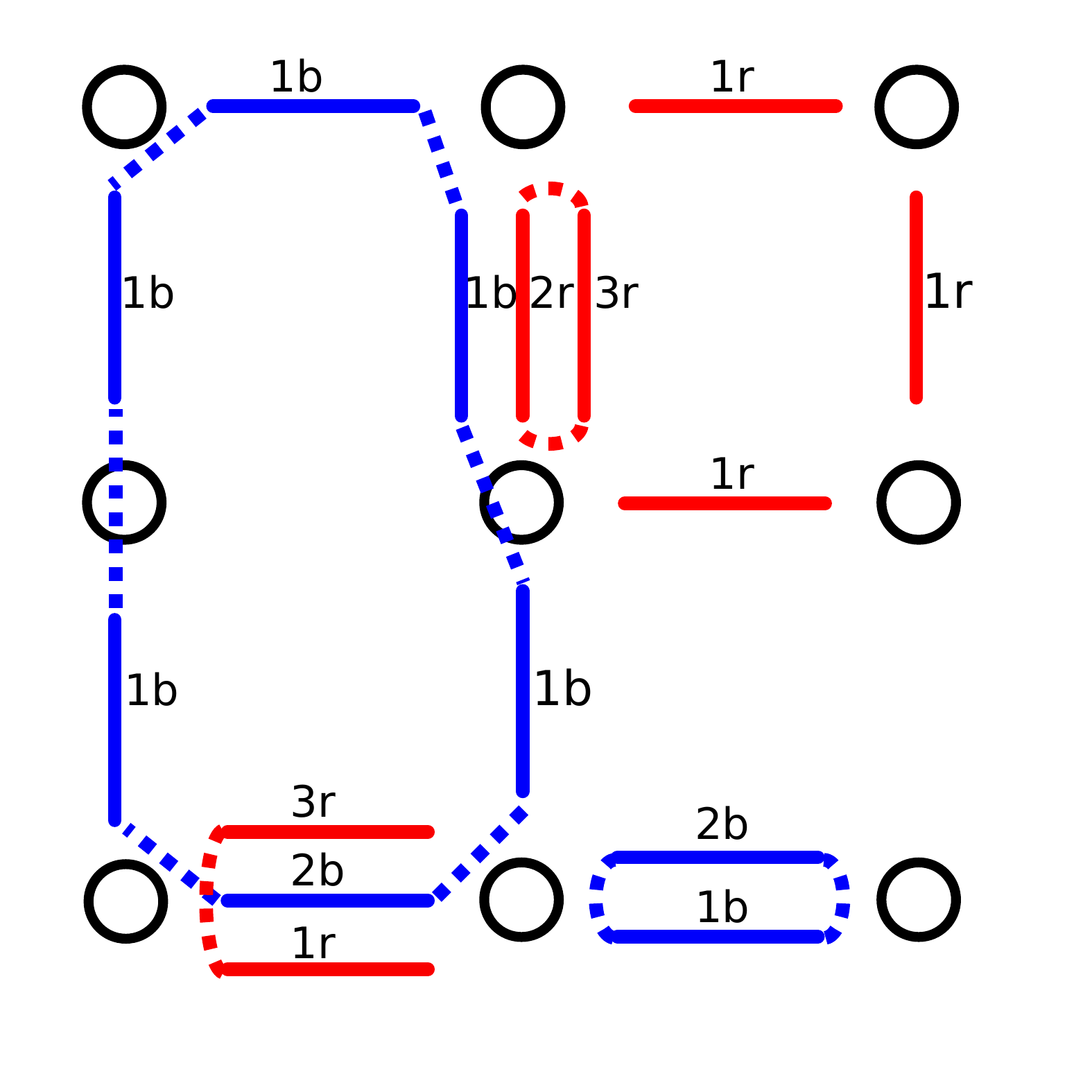}
\centering
\caption{
A configuration $w = (m, c, \gamma) \in \mathcal{W}_{\mathcal{G}}$, where $\mathcal{G}$ corresponds to the graph  $\{1, 2,  3\} \times \{1,2,3\}$ with edges connecting nearest neighbours and the lowest leftmost vertex corresponds to $(1,1)$.
On every edge $e$, the links are ordered and receive a label from $1$ to $m_e$.
In the figure, the numbers 1, 2, ...  are used for the identification of the links
and the letters $b$ and $r$ are used for the colours which are assigned to the links by $c$ (we assume that $N=2$ and that each link might be either blue or red).  Paired links are connected by a dotted line.
For example, the first link on the edge connecting the vertices $(1,1)$, $(2,1)$
is coloured  red,  it is paired at $(1,1)$ with the third link on the same edge and it is unpaired at $(2,1)$. Moreover, both links touching the vertex $(3,3)$ are red and they are unpaired at $(3,3)$.
Finally, no link is on the edge which connects the vertices $(1,2)$ and $(2,2)$. 
}
\label{Fig:pairingexample}
\end{figure}

Given $m \in \mathcal{M}_{\Gcal}$, a \textit{colouring} $c = (c_e)_{e \in \Ecal}$, with $c_e : \{1, \ldots, m_e\} \mapsto [N]$ is a function which assigns an integer (colour) in $[N]$ to each link. More precisely, we use $(e, p)$ to represent the $p$th link on the edge $e$, with $p \in \{1, \ldots, m_e\}$, 
and we let  $c\big ( (e,p) \big ) \in [N]$ be the colour of the $p^{th}$ link on $e \in \Ecal$.
A link with colour $i \in [N]$ is called an $i$-link. For $e\in \Ecal$ and $i\in[N]$, we denote by $m^i_e$ the number of $i$-links on $e$.
We let $\mathcal{C}_{\Gcal}(m)$ be the set of possible colourings $c = (c_e)_{e \in \Ecal}$ for $m$.

Given a {link configuration} $m \in \mathcal{M}_{\Gcal}$, and a colouring $c \in \mathcal{C}_{\Gcal}(m)$,  we say $ \pi = (\pi_x)_{x \in \Vcal}$ is a  \emph{pairing} of $(m,c)$ if, for each $x\in \Vcal$, $\pi_x$ pairs links on the edges incident to $x$ in such a way that if two links are paired, then they have the same colour. 
A link incident to $x$ is paired to at most one other link incident to $x$ and, possibly, it is not paired to any link at $x$
(formally, $\pi_x$ is a partition of the set of links touching $x$ so that  each element of the partition is a set containing either only one link or two links of the same colour).
We say two links are \emph{paired} if there is an $x\in \Vcal$ such that the links are paired at $x$. A link can be paired to at most two other links, one at each end point of its edge. We remark that, by definition, a link cannot be paired to itself. Denote by $\mathcal{P}_{\Gcal}(m,c)$ the set of all such pairings for $m\in\mathcal{M}_{\Gcal}$ and $c \in \mathcal{C}_{\Gcal}(m)$. Note that $ \mathcal{P}_{\Gcal}(m, c)$ generally has many elements, corresponding to the number of ways the links can be paired.

A \emph{wire configuration} on $\Gcal$ is an element $w = ( m, c, \pi)$
such that $m \in \mathcal{M}_{\Gcal}$, $c \in \mathcal{C}_{\Gcal}(m)$,  and $\pi \in \mathcal{P}_{\Gcal}(m, c)$. Let $\mathcal{W}_{\Gcal}$ be the set of  wire configurations on $\Gcal$. As we can see from the example in Figure \ref{Fig:pairingexample}, it follows that any $w \in \mathcal{W}_{\Gcal}$ can be viewed as a collection of closed or open \textit{paths}, open paths will be called \textit{walks} and closed paths will be called \textit{loops} (see the Appendix for a formal definition of such objects). For example, Figure \ref{Fig:pairingexample} presents three loops and four walks.
If the links of a loop or a walk have colour $i$, we might refer to it as an $i$-loop or an $i$-walk respectively.
By a slight abuse of notation, we will also view 
$m : \mathcal{W}_{\Gcal} \mapsto \mathcal{M}_{\Gcal}$
as a function such that, for  $w^{\prime} = ( m^\prime, c^{
\prime}, \pi^{\prime})$, $m(w^\prime) = m^\prime$.

Let $u^i_x(w)$ be the  \textit{number of $i$-links incident to $x$} which are unpaired at $x$ (i.e. the number of walk end points at $x$). Let $v^i_x(w)$ be the number of $i$-links incident to $x$ which are paired to another link at $x$, divided by two (i.e. the number of times a path passes through $x$), 
\begin{align}\label{eq:numerofihits}
v^i_x(w) & :=  \frac{1}{2} \Big ( \sum\limits_{y \sim x} m^i_{ \{x,y\}  }   - \,  u^i_x(w) \Big ).
\end{align}
Moreover, let $n_x^i(w) : = v^i_x(w)  + u^i_x(w)$ be the  \textit{local time of $i$-objects at $x$}. Unpaired  end-points of links touching  $x$ and pairs of paired links touching $x$ both contribute +1 to the local time.

Let $\mathcal{G} = (\mathcal{V}, \mathcal{E})$ be an arbitrary finite simple undirected graph, we want to  introduce a \textit{representation} for the Spin $O(N)$ model on $\mathcal{G}$ in the presence of an external magnetic field. Hence, we introduce a \emph{ghost vertex}  $g\notin \Vcal$, and  the graph $G=(V,E)$,  
with 
$$
V=\Vcal\cup\{g\} \quad \quad E=\Ecal\cup \big\{ \{x,g\} \, : \, x\in \Vcal\big\}.$$
We call any edge which is incident to the ghost vertex a \textit{ghost edge} and  any edge which is not an  \textit{original edge}.
We also call any vertex in $\mathcal{V}$ an \textit{original vertex}.

\begin{defn}\label{def:measure}
We let $\mathcal{W}^\prime_G$ be the set of configurations 
$w \in \mathcal{W}_G$  such that 
$v^N_g(w) = n^1_g(w) = \ldots = n^{N-1}_g(w) = 0$ and
$u^2_x(w) = \ldots = u^{N}_x(w)=  0$ for every $x \in \Vcal$.
Given $N \in \mathbb{N}_{>0}$, $\beta \in \mathbb{R}_{\geq 0}$, $h\in\R$, we  define the non-negative (not necessarily probability) measure $\mu_{G, N, \beta, h}$ on $\mathcal{W}^\prime_G$
as follows, for any $w = (m, c, \pi) \in \mathcal{W}^\prime_G$, 
\begin{equation}\label{eq:weight3}
\quad \mu_{G, N, \beta, h}(w) : =\Big  (\prod_{e \in \mathcal{E}} \frac{\beta^{m_e}}{m_e!} \Big ) \, \,
\Big  (\prod_{x \in \mathcal{V}} \frac{h^{m_{ \{x,g\} }}}{m_{\{x,g\}}!} \Big )  \, \, 
 \Big ( \prod_{x \in \Vcal} U_x(w) \Big ),
\end{equation}
where $U_x(w) := \mathcal{U}(n_x(w))$, with 
\begin{equation}\label{eq:weightfunction}
\forall r \in \mathbb{N} \quad \quad \mathcal{U}(r) : =   \frac{  \Gamma( \frac{N}{2} )  }{2^r  \, \, \Gamma( r + \frac{N}{2}   )},
\end{equation}
and $n_x=\sum_in^i_x$ is the local time at $x$.
Given a function
$f : \mathcal{W}^\prime_{G} \rightarrow \mathbb{R}$, we use the same notation for the expectation of $f$ under $\mu_{G,N, \beta, h}$,
$
\mu_{G,N, \beta, h}(f) : = \sum\limits_{w  \in \mathcal{W}^\prime_{G}}  f(w) \, \, \mu_{G,N, \beta, h}(w).
$
\end{defn}
In other words, at the ghost vertex  only unpaired end-points of $N$-links are allowed, while at the original vertices  either paired links of any colour or unpaired $1$-links are allowed, and a weight
 which depends on the local time is assigned. Thus, any realisation $w \in \mathcal{W}^\prime_G$ consists of open paths of colour $N$ with both end points at the ghost vertex, open paths of colour $1$ with end points at original vertices, and closed paths of any colour. 
Closed paths of colour $1$, $\ldots$, $N-1$ lie entirely in $\mathcal{G} \subset G$, as do open paths of colour $1$.

Notice that, because all open paths necessarily have two end-points, the power of $h$ in $\mu_{G, N, \beta, h}(w)$ is always even, hence the results we obtain for $h$ and $-h$ will be identical. With this in mind we will in often take $h>0$.

The central quantity of interest is the  \emph{two-point function}.
For the definition of the two-point function we will allow  only one walk of colour $1$.
\begin{defn}\label{def:twopointpaths}
For $A\subset \Vcal$, define $\mathcal{S}(A)$ to be the set of configurations $w \in \mathcal{W}^\prime_{G}$ such that $u_z^1(w) = 1$ for every $z\in A$ and $u_z^1(w) = 0$ for every $z \in \Vcal \setminus A$. 
 We define 
  $Z_{G,N, \beta, h}(A)=\mu_{G,N, \beta, h}(  \mathcal{S}(A))$
  and 
$
Z^{\ell}_{G,N,\beta, h}=\mu_{G,N, \beta, h} \big ( \mathcal{S}(\emptyset) \big ).
$
 Finally, we define the \emph{point-to-point correlation functions} by, 
\begin{align*}
\mathbb{G}_{G,N, \beta, h}(A)   : = \frac{Z_{G,N, \beta, h}(A)}{Z^{\ell}_{G,N,\beta, h}}. 
\end{align*}
We call the cases where $|A|=2$ \emph{two-point functions}. 
When $A=\{x,y\}$ for $x\neq y$ we write $Z_{G,N, \beta, h}(x,y)$ and $\mathbb{G}_{G, N, \beta}(x,y)$ 
for $Z_{G,N, \beta, h}(A)$ and $\mathbb{G}_{G,N, \beta}(A)$ respectively.
We also write $\mathcal{S}_G$ for  $\mathcal{S}(\emptyset)$.
\end{defn}

\subsection{Equivalence of two-point functions and the `colour-switch' lemma}

The next proposition connects the correlation function which was defined above to the spin correlations of the Spin $O(N)$ model.

\begin{prop}\label{prop:looprep}
Let $\Gcal=(\Vcal,\Ecal)$ be an undirected, finite, simple graph and let $G=(V,E)$ be obtained from $\Gcal$ by adding a ghost vertex $g$ as described above. Let $N\in\N_{>0}$, $\beta\geq 0$ and $h\in\R$. 
We have that,
\begin{equation}
\mathbb{G}_{G,N, \beta, h}(A) =\bigg\langle \prod_{x\in A}\varphi^1_x\bigg\rangle^{spin}_{\Gcal,N,\beta,h}.
\end{equation}
\end{prop}
\begin{proof}
The proof is very similar to \cite[Proposition 2.3]{L-T}. To begin, for $A\subset \Vcal$ we define 
\begin{equation}
Z^{spin}_{\Gcal,N,\beta,h}(A):=Z^{spin}_{\Gcal,N,\beta,h}\bigg\langle \prod_{x\in A}\varphi^1_x\bigg\rangle^{spin}_{\Gcal,N,\beta,h}=\int_{\Omega_{\Gcal,N}}\mathrm{d}\varphi \left(\prod_{x\in A} \varphi^1_x\right)e^{-H^{spin}_{\Gcal,N,\beta,h}(\varphi)}.
\end{equation}
Now we expand the exponential term, we will define $\varphi_g:=(0,\dots,0,1)$ in order to have a consistent notation and cleaner expressions in the expansion. The reader should understand the the `spin' at $g$ is fixed to $(0,\dots,0,1)$. For convenience we will define a coupling parameter that incorporates $\beta$ and $h$. For $\{x,y\}\in E$ and $i\in[N]$
\begin{equation}
J^i_{\{x,y\}}=\begin{cases} \beta & \text{ if }\{x,y\}\in\Ecal, \\ h & \text{ if } g\in\{x,y\} \text{ and }i=N, \\ 0 & \text{ otherwise.} \end{cases}
\end{equation}

This will enable us to write our expansion in terms of a single variable, $J^i_e$, instead of having to constantly differentiate between different cases. To begin we write the exponential term as,
\begin{equation}\label{eq:expexpand}
\exp\bigg\{\sum_{\{x,y\}\in E}\sum_{i=1}^NJ^i_{\{x,y\}}\varphi^i_x\varphi^i_y\bigg\}=\prod_{\{x,y\}\in E}\prod_{i=1}^Ne^{J^i_{\{x,y\}} \varphi^i_x\varphi^i_y}.
\end{equation}
and expand 
\begin{equation}
e^{J^i_{\{x,y\}}\varphi^i_x\varphi^i_y}=\sum_{m^i_{\{x,y\}}\geq 0}\frac{(J^i_{\{x,y\}})^{m^i_{\{x,y\}}}}{m^i_{\{x,y\}}!}(\varphi^i_x\varphi^i_y)^{m^i_{\{x,y\}}}.
\end{equation}
 For $B\subset \Vcal$ we define sets
\begin{align}
\widetilde{\Mcal}_{G}(B)=&\bigg\{m\in\Mcal_{G}\, : \, \forall x\in B\, \sum_{e\in E:x\in e}m_e\in 2\N+1,\, \forall x\in V\setminus B\, \sum_{e\in E:x\in e}m_e\in2\N \bigg\}
\\
\Mcal_{\Gcal}(B)=&\widetilde{\Mcal}_{G}(B)\cap \bigg\{m\in\Mcal_{G}\, : \, \sum_{x\in \Vcal}m_{\{x,g\}}=0\bigg\}.
\end{align}
We also define $q^i_x(m)=\sum_{e\ni x}m^i_e$ and $q_x(m)=\sum_{i=1}^N q^i_x(m)$. We have
\begin{equation}
\begin{aligned}
Z^{spin}_{\Gcal,N,\beta,h}&(A)=\sum_{m^1\in \Mcal_G(A)}\sum_{m^2,\dots,m^{N-1}\in\Mcal_G(\emptyset)}\sum_{m^N\in\widetilde{\Mcal}_G(\emptyset)}\left[\prod_{e\in E}\left(\prod_{i=1}^N\frac{(J^i_e)^{m^i_e}}{m^i_e!}\right)\right]
\\
&\int_{\Omega_{\Gcal,N}}\mathrm{d}\varphi\left(\prod_{x\in A}(\varphi^i_x)^{q^i_x+1}(\varphi^2_x)^{q^2_x}\dots (\varphi^N_x)^{q^N_x}\right)\left(\prod_{x\in V\setminus A}(\varphi^i_x)^{q^i_x}(\varphi^2_x)^{q^2_x}\dots (\varphi^N_x)^{q^N_x}\right).
\end{aligned}
\end{equation}

Now we use the following identity from \cite[Appendix A]{Chayes}
\begin{equation}\label{eq:spinintegral}
\int_{\mathbb{S}^{N-1}}(\varphi^1)^{n_1}\dots (\varphi^N)^{n_N}\mathrm{d}\varphi=\begin{cases}
\frac{\Gamma\big(\tfrac{N}{2}\big)\prod_{i=1}^N (n_i-1)!!}{2^{\frac{n}{2}}\Gamma\big(\tfrac{n+N}{2}\big)}& \text{if } n_i\in2\N\text{ for }i\in[N], \\
 0 &\text{otherwise},
 \end{cases}
\end{equation}
with $n = \sum_{i=1}^N n_i$.
Additionally, we sum over uncoloured link configurations and over ways to distribute the colours of these configurations to obtain
\begin{equation}
\begin{aligned}
&Z^{spin}_{\Gcal,N,\beta,h}(A)=\sum_{m\in \widetilde\Mcal_G(A)}\left(\prod_{e\in E}\frac{1}{m_e!}\right)\sum_{\substack{m^1\in\Mcal_G(A),m^N\in \widetilde{\Mcal}_G(\emptyset)\\ m^2,\dots,m^N\in \Mcal_G(\emptyset)\\\sum_{i=1}^Nm^i=m}}\left(\prod_{e\in E}\frac{m_e!}{m^1_e!\dots m^N_e!}\prod_{i=1}^N(J^i_e)^{m^i_e}\right)
\\
&\left(\prod_{x\in A}\frac{\Gamma\big(\tfrac{N}{2}\big)}{2^{(q_x+1)/2}\Gamma\big(\tfrac{q_x+1+N}{2}\big)} q^1_x!!\prod_{i=2}^N(q^i_x-1)!!\right)
\left(\prod_{x\in V\setminus A}\frac{\Gamma\big(\tfrac{N}{2}\big)}{2^{q_x/2}\Gamma\big(\tfrac{q_x+N}{2}\big)}\prod_{i=1}^N(q^i_x-1)!!\right).
\end{aligned}
\end{equation}

Now if $q\in 2\N$ then $(q-1)!!$ is the number of pairings of $q$ objects, whereas if $q\in2\N+1$ then $q!!$ is the number of pairings of $q$ objects that leaves one object on its own (i.e. there are $(q-1)/2$ tuples and one single object). 

For $m\in\widetilde\Mcal_{G}(A)$ let $\Ccal_{G}(m,A)$ be the set of colourings such that for every $x\in V$ the number of 1-links incident to $x$ is odd if $x\in A$ and is even otherwise and all links incident to $g$ are $N$-links. For $i\in\{2,\dots,N\}$ there are an even number of $i$-links incident to $x$ for every $x\in\Vcal$.

Further, for $m\in\widetilde\Mcal_{\Gcal}(A)$ and $c\in\Ccal_{\Gcal}(m,A)$ let $\Pcal_{\Gcal}(m,c, A)$ be the set of pairings such that there is precisely one unpaired 1-link at each $x\in A$ (and no other unpaired 1-links), additionally for $i\in\{2,\dots,N\}$ and every $x\in \Vcal$ each $i$-link incident to $x$ is paired at $x$ and no $N$-links incident to $g$ are paired. Given such a triple $(m,c,\pi)$ and $x\in V$ let $n_x(m,c,\pi)=\sum_{i=1}^Nn^i_x(m,c,\pi)$ be the local time  at the vertex $x$, we have
\begin{equation}
\begin{aligned}
&Z^{spin}_{\Gcal,N,\beta,h}(A) = \sum_{m\in \widetilde\Mcal_G(A)}\sum_{c\in \Ccal_{G}(m,A)}\left(\prod_{e\in E}\frac{1}{m_e!}\prod_{i=1}^N(J^i_e)^{m^i_e}\right)
\\
&\sum_{\pi\in \Pcal_G(m,c, A)}\left(\prod_{x\in A}\frac{\Gamma\big(\tfrac{N}{2}\big)}{2^{n_x(m,c,\pi)}\Gamma\big(n_x(m,c,\pi)+\tfrac{N}{2}\big)}\right)\left(\prod_{x\in V\setminus A}\frac{\Gamma\big(\tfrac{N}{2}\big)}{2^{n_x(m,c,\pi)}\Gamma\big(n_x(m,c,\pi)+\tfrac{N}{2}\big)}\right).
\end{aligned}
\end{equation}
 We used that, if $w=(m,c,\pi)\in\Wcal_{G}$ and $q_x$ links are incident to $x$ then if $x\in A$, $q_x+1=2n_x(w)$. Similarly if $x\in\Vcal\setminus A$ then $q_x=2n_x(w)$.

Now if we define $U_x$ as in Definition \ref{def:measure},
we recall the definition of $\mathcal{W}_G^\prime$
and  perform the same expansion for $Z^{spin}_{\Gcal,N,\beta,h}=Z^{spin}_{\Gcal,N,\beta,h}(\emptyset)$ we have the result.
\end{proof}

The previous proposition connects the spin-spin correlation of the Spin $O(N)$ model to the correlation functions of the random path model.  The starting point of our analysis is Lemma \ref{lem:colourswitchlem} below,
which connects the two-point correlation function to the expected number of $N$-walks with extremal links (defined below Definition \ref{def:loopwalkmeasure}) on $\{x,g\}$ and $\{y,g\}$
in a random path configuration with loops of any colour and $N$-walks with both end points at the ghost vertex (and no walks of colour 1).
The next definition introduces the probability measure and expectation which describes such a random path model.
\begin{defn}\label{def:loopwalkmeasure} 
We define the probability measure on $\mathcal{S}_G \subset \mathcal{W}^\prime_G,$
\begin{equation}\label{eq:probmeasure}
\forall w  \in \mathcal{S}_G  \quad  \quad \mathbb{P}_{G, N, \beta, h} \big  (  w  ) \, : = \, \frac{1}{  Z^{\ell}_{G,N,\beta, h }} \, \, 
\mu_{G,N, \beta, h} ( w ),
\end{equation}
 and we denote by $\mathbb{E}_{G, N, \beta, h}$ the expectation with respect to $\mathbb{P}_{G, N, \beta, h}$. 
\end{defn}
We now introduce the definition of an extremal link. A link is called \textit{extremal}
if at least one of its end-points is unpaired. Given a walk consisting of least two links, we call its two links which have an unpaired end-point the \textit{extremal links of the walk}. 
If a walk has its two extremal links on the edges $e_1, e_2 \in E$, we write that it is $(e_1, e_2)$ - extremal.
Notice that $N$-walks  have both end points at the ghost vertex $\mathbb{P}_{G, N, \beta, h}$ - almost surely, hence they  have at least two links.

The next lemma is key for our approach. 

\begin{lem}[\textbf{Colour-Switch lemma}]\label{lem:colourswitchlem}
Let $N\in\N_{\geq 2}$, $\beta\geq 0$ and $h \in \R\setminus\{0\}$.
 Choose an arbitrary pair of distinct vertices $x, y \in \mathcal{V}$.
We let $M_{x,y}$ be the number of $N$-walks which have one extremal link on  $\{x,g\}$ and the other extremal link on $\{y,g\}$. Then,
\begin{equation}\label{eq:colourswitch}
\mathbb{G}_{G,N, \beta, h}(x,y) =\frac{1}{h^2} 
\mathbb{E}_{G,N, \beta, h} \Big ( 
M_{x,y}\Big).
\end{equation}
\end{lem}
\begin{proof}
\begin{figure}
\includegraphics[scale=0.38]{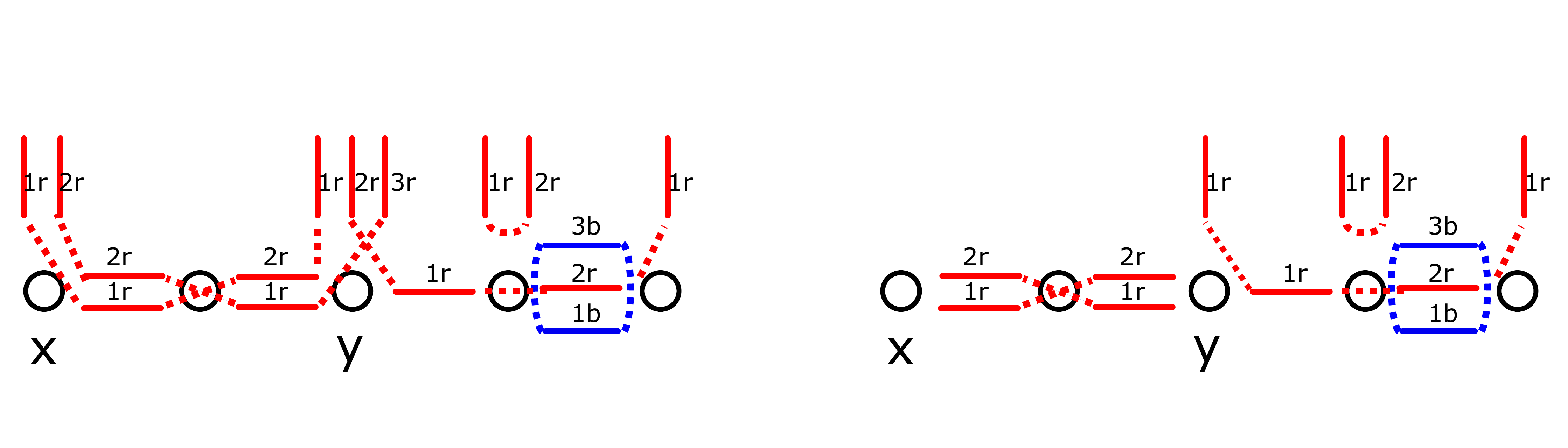}
\centering
\caption{
We suppose that $N = 2$, we represent the colour $1$ by blue and the colour $N=2$ by red, we assume that $\mathcal{G}$ is a connected subset of $\mathbb{Z}$ with five vertices and edges connecting nearest neighbour vertices and the vertical links are on edges connecting the vertices of $\Gcal$  to $g$, the ghost vertex is not represented in the figure. 
{\textit{Left:}} A configuration $w \in \mathcal{S}^N( \{x,y\}  )$
such that $M_{x,y}(w) = 2$ (see the definitions in the proof of Lemma \ref{lem:colourswitchlem}). 
\textit{Right:} A configuration $F_N(w)$, which is obtained from $w$ by removing the extremal links of the two
$(\{x,g\},\{y,g\})$ - extremal  walks  and by leaving unpaired the links to which such external links were paired.
}
\label{Fig:mapexample}
\end{figure}
We let  $\Scal^N(\{x,y\})\subset \Scal_G$ be the set of configurations in $\Scal_G$ with no $1$-walks and with at least one $N$-walk whose extremal links are on the edges $\{x,g\}$ and $\{y,g\}$, recall that $\Scal(\{x,y\})$ is the set of configurations with a unique $1$-walk having $x$ and $y$ as end-points. The proof consists of
partitioning the sets $\Scal^N(\{x,y\})$ and  $\Scal(\{x,y\})$, identifying a bijection between
the elements of the partition in $\Scal^N(\{x,y\})$ and  those in $\Scal(\{x,y\})$, and
comparing the  weights of such elements.

To begin, we define a map $F_N : \Scal^N(\{x,y\}) \mapsto \mathcal{W}_G$ which acts by removing  the extremal links 
of  the $(\{x,g\},\{y,g\})$ - extremal walks 
and by leaving unpaired the links to which these extremal links were paired,  as in the example in Figure \ref{Fig:mapexample}.

Further, we define a map $F_1 : \Scal(\{x,y\}) \mapsto \mathcal{W}_G$ which acts by first applying $F_N$ to $w \in \Scal(\{x,y\})$ (in the analogous way as on $ \Scal^N(\{x,y\})$)
and then by changing to $N$ the colour of all the links belonging to the unique $1$-walk with end-points $x$ and $y$.
We note that
\begin{equation}\label{eq:imageset}
R: = F_1 \big (   \Scal(\{x,y\}) \big ) = F_N\big (  \Scal^N(\{x,y\}) \big ) \subset \mathcal{W}_G
\end{equation}
The configurations in $R$ have at least one $N$-walk having $x$ and $y$ as end-points. 
For any $w \in R$, we let $M^r_{x,y}(w)$ be the number of $N$-walks with end-points $x$ and $y$ in $w$.
We note that, for any $w, w^\prime \in R$ such that $w \neq w^\prime$, 
\begin{equation}\label{eq:disjointeness}
F_1^{-1}(w) \cap F_1^{-1}(w^\prime) = \emptyset \quad \quad F_N^{-1}(w) \cap F_N^{-1}(w^\prime) = \emptyset.
\end{equation}
Here, for any $w \in R$,  $F_N^{-1}(w)$ corresponds to the set of configurations which are obtained from $w$ by inserting $M^r_{x,y}(w)$  $N$-links
on $\{x,g\}$ and $\{y,g\}$ and by pairing them at $x$ and $y$ to the 
links of the  $M^r_{x,y}(w)$  $N$-walks
with end-points $x$ and $y$ in some arbitrary manner.
Similarly,  $F_1^{-1}(w)$ corresponds to the set of configurations which are obtained from $w$ by choosing one of the $N$-walks with end-points $x$ and $y$, turning the colour of all its links to $1$, inserting $M^r_{x,y}(w)-1$  $N$-links on $\{x,g\}$ and $\{y,g\}$ and pairing them at $x$ and $y$ to the links of the  $M^r_{ x,y }(w)-1$ remaining $N$-walks with end-points $x$ and $y$ in some arbitrary manner. It follows that, if $w \neq w^\prime \in R$, then each configuration in  $F_1^{-1}(w)$ differs from each configuration in $F_1^{-1}(w^\prime)$ and the same holds for $F_N^{-1}$, giving (\ref{eq:disjointeness}).

We note that, for any $w^\prime = (m^\prime, c^\prime, \pi^\prime    ) \in R$, we have that,
\begin{multline}\label{eq:relation1}
\big |  F^{-1}_N(  w^\prime   )     \big |  =    \binom{M^r_{x,y}(w^\prime) + m^\prime_{\{x,g\}}}{M^r_{x,y}(w^\prime)  } \,   \binom{M^r_{x,y}(w^\prime) + m^\prime_{\{y,g\}}}{M^r_{x,y}(w^\prime)  } \,  (M^r_{x,y}(w^\prime)!)^2 \\ =  \frac{ (M^r_{x,y}(w^\prime) + m^\prime_{\{x,g\}} )!}{ m^\prime_{\{x,g\}}!   }    \,   \frac{ (M^r_{x,y}(w^\prime) + m^\prime_{\{y,g\}})! }{ m^\prime_{\{y,g\}}!  }   
\end{multline}
where the first two factors in the right-hand side of the first identity correspond to the 
number of ways $M^r_{x,y}(w^\prime)$ $N$-links can be inserted on
$\{x,g\}$ and on  $\{y,g\}$ among the ones already present in  $w^\prime$ and the  third factor  corresponds to the number of ways such new links can be paired at $x$ and $y$ 
to the $M^r_{x,y}(w^\prime)$ links of the $N$-walks of $w^\prime$ with end-points $x$ and $y$.
Similarly, we obtain that, 
\begin{multline}\label{eq:relation2}
\big |  F^{-1}_1(  w^\prime)       \big | = M^r_{x,y}(w^\prime) \, 
\binom{M^r_{x,y}(w^\prime) -1 + m^\prime_{\{x,g\}}}{M^r_{x,y}(w^\prime) -1  } \,   \binom{M^r_{x,y}(w^\prime) -1 + m^\prime_{\{y,g\}}}{M^r_{x,y}(w^\prime) -1 } \,  ((M^r_{x,y}(w^\prime)-1)!)^2\\ 
 =  M^r_{x,y}(w^\prime)  \, \frac{ (M^r_{x,y}(w^\prime)-1 + m^\prime_{\{x,g\}} )!}{ m^\prime_{\{x,g\}}!   }    \,   \frac{ (M^r_{x,y}(w^\prime) -1 + m^\prime_{\{y,g\}})! }{ m^\prime_{\{y,g\}}!  } , 
\end{multline}
where the first factor in the right-hand side of the first identity 
corresponds to the number of ways for choosing which of the $M^r_{x,y}(w^\prime)$ $N$-walks is turned into an $1$-walk and the remaining weights are analogous to those in the previous display.

We now note that, from Definition \ref{def:measure}, for any
$w^\prime = (m^\prime, c^\prime, \pi^\prime) \in R$ and any
 $w = (m, c, \pi) \in F^{-1}_N(w^\prime)$, we have that,
\begin{equation}\label{eq:relation3}
\mu_{G, N, \beta, h }(w) =\Big  (\prod_{e \in \Ecal} \frac{\beta^{m^\prime_e}}{m^\prime_e!} \Big ) \, 
\Big  (\prod_{z \in \Vcal \setminus \{x,y\} } \frac{h^{m^\prime_{ \{z,g\} }}}{m^\prime_{\{z,g\}}!} \Big ) \,
 \Big ( \prod_{z \in \Vcal} U_z(w^\prime) \Big )\,
\Big  (\prod_{z \in  \{x,y\} } \frac{h^{M_{x,y}(w^\prime) +m^\prime_{ \{z,g\}}} }{( M_{x,y}(w^\prime) + m^\prime_{\{z,g\}} )!} \Big )  
\end{equation}
and, similarly, that 
for any
$w^\prime = (m^\prime, c^\prime, \pi^\prime) \in R$ and any
 $w = (m, c, \pi) \in F^{-1}_1(w^\prime)$ we have that,
\begin{equation}\label{eq:relation4}
\mu_{G, N, \beta, h }(w) =\Big  (\prod_{e \in \Ecal} \frac{\beta^{m^\prime_e}}{m^\prime_e!} \Big ) \, 
\Big  (\prod_{z \in \Vcal \setminus \{x,y\} } \frac{h^{m^\prime_{ \{z,g\} }}}{m^\prime_{\{z,g\}}!} \Big )  \,
 \Big ( \prod_{z \in \Vcal} U_z(w^\prime) \Big ) \, 
 \Big  (\prod_{z \in  \{x,y\} } \frac{h^{M^r_{x,y}(w^\prime) -1 + m^\prime_{ \{z,g\}}}}{(M^r_{x,y}(w^\prime) -1+ m^\prime_{\{z,g\}} )!} \Big ).
\end{equation}
Note that the two previous  displays differ from each other only in the last factor. Thus, combining  the last four displays we deduce that,
for any $w^\prime \in R$,
\begin{equation}\label{eq:relation5}
\mu_{G, N, \beta, h }
\Big  (  F^{-1}_1(w^\prime)   \Big ) =
\frac{1}{h^2} \,
M^r( w^\prime  ) \,  \mu_{G, N, \beta, h } \,
\Big  (  F^{-1}_N(w^\prime)   \Big ).
\end{equation}
The previous identity can be deduced from 
(\ref{eq:relation1}),  (\ref{eq:relation2}),
(\ref{eq:relation3}), (\ref{eq:relation4})
since (\ref{eq:relation1}) and  (\ref{eq:relation2})
give the cardinalities of the sets  in the left- and right-hand side of (\ref{eq:relation5})  respectively
and   (\ref{eq:relation3}), (\ref{eq:relation4})
 give the weight of each  element of the sets.
 Thus, we obtain that,
\begin{align*}
&  Z_{G, N, \beta, h}(x,y)
  =
 \sum_{w\in \Scal(\{x,y\}) }\mu_{G,N,\beta, h}(w) 
 =
\sum_{w^\prime\in R }\mu_{G,N,\beta, h} \big ( F^{-1}_1(w^\prime) \big ) 
\\ 
& =
\frac{1}{h^2} 
\sum_{w^\prime\in R }\mu_{G,N,\beta, h} \big ( F^{-1}_N (w^\prime) \big )   M^r_{x,y}(w^\prime)
  =
\frac{1}{h^2} \, \sum_{w\in \Scal^N(\{x,y\}) }\mu_{G,N,\beta, h}  ( w ) \, M_{x,y}(w)
  \\ &  =  
\frac{1}{h^2} \, \sum_{w\in \Scal_G }\mu_{G,N,\beta, h}  ( w ) \, M_{x,y}(w)
=
Z^\ell_{G,N,\beta, h} \, \frac{1}{h^2}  \, \mathbb{E}_{G,N, \beta, h} 
\big ( M_{x,y} \big),
\end{align*}
where 
for the second identity we used  (\ref{eq:disjointeness}),
for the third identity we used
(\ref{eq:relation5}),
for the fourth identity we used again  (\ref{eq:disjointeness})
and the fact that, for any $w^\prime \in R$
and $w \in F^{-1}_N(w^\prime)$,
$M_{x,y}^r(w^\prime) = M_{x,y}(w)$,
for the last identity we used Definitions
 \ref{def:twopointpaths}  and \ref{def:loopwalkmeasure}.
 From the last expression and Definition   \ref{def:twopointpaths}
 we deduce (\ref{eq:colourswitch}) and conclude the proof.
\end{proof}

\section{A bound on local times}\label{sec:linkbounds}
The next lemma provides an upper bound for the joint distributions of local times of vertices when the maximum degree of $\Gcal$, $d^*_{\Gcal}$, is finite. 
Since the measure in Definition \ref{def:measure} is invariant with respect to a sign inversion of the external magnetic field (the number of unpaired end-points of $N$-links can only be even),  we can, without loss of generality, take $h \geq 0$.

\textbf{Further notation.} Recall that for an arbitrary finite directed graph $\mathcal{G} = (\mathcal{V}, \mathcal{E})$ we denote by $G = (V,E)$ the graph obtained from $\mathcal{G}$ by adding a ghost vertex, as described above Definition \ref{def:measure}. 
Given a set of original vertices $A \subset  \mathcal{V}$, we introduce the following notation for the complement and for the edge and vertex boundaries of $A$.  We define $A^c : =  \mathcal{V} \setminus A$, note that this set only contains original vertices. We let $E_{A}$ be the set of edges in $E$ which have at least one end-point in  $A$ (these can be original or ghost), we let $\partial E_A$ be the set of edges in $E$ which have one end-point in $A$ and the other end-point in $A^c$, and we let $E^g_A$ be the set of edges which connect a vertex in $A$ to the ghost vertex. We denote the external boundary of $A$ by $\partial^e A$ - the set of vertices in $x \in A^c$ which have a neighbour in $A$ - and the internal boundary by $\partial^i A$ - the set of vertices in $A$ which have a neighbour in $A^c$. 
By definition, $\partial^i A, \partial^e A \subset \mathcal{V}$.

\hypertarget{c1}{
\begin{lem}\label{lemma:linkcardinalityset}
Let $N\in\N_{>0}$, $\beta, h  \geq 0$. For any $k\in \N$ there exists $c_1 = c_1(d^*_\Gcal, k, N, \beta, h) \in (0, \infty)$ satisfying $\lim_{k \rightarrow \infty} c_1 = 0$  such that,
for any set $A \subset \mathcal{V}$ and $z\in A$,
\begin{align}
\label{eq:Anbound}
\P_{G,N,\beta,h}(\forall\, x\in A ,\, n_x \geq k) & \leq  c_1^{|A|}, \\
\label{eq:Bnbound}
\E_{G,N,\beta,h}\big(m_{\{z,g\}}\1_{\{\forall x\in A \, n_x \geq k\}}\big)  & \leq h\, c_1^{|A|}.
\end{align}
\end{lem}
}
\begin{proof}
To begin, we define the measure in $\mathcal{S}_G$, 
$$
\forall w \in \mathcal{S}_G \quad \quad 
\tilde \mu(w) : = \mathbb{P}_{G, N, \beta, h  }( w) \, Z^\ell_{G, N, \beta, h  }.
$$
We define by 
$\Sigma_G$  
the set of elements 
$\xi= (\xi^i_e)_{e \in E, i \in [N]} \in \{0,1\}^{E \times [N] }$
such that, for any $x \in \Vcal$ and $i \in [N]$, we have that
$\sum_{ y \sim x} \xi^i_{\{x,y\}} \in 2 \mathbb{N}$.
We have that,
\begin{equation}\label{eq:splittingmultline}
\tilde  \mu
\big ( \forall x \in A,    n_x  \geq  k    \big ) \\
=  \sum\limits_{ \substack{ \xi \in \Sigma_G  }}
\tilde  \mu \big (   \big\{  \forall x \in A, n_x \geq k\big\} \cap \big\{ \forall i \in [N],  \forall e \in E, 
m^i_e(w) \in  2 \mathbb{N} + \xi^i_e  \big\}     \big ).
\end{equation}
For $(n^1,\dots,n^N)\in \N^N$ we define the quantity,
\begin{align}\label{eq:ignoresites}
\mathcal{X}(  n^1, \ldots,  n^N) &  : = 
\int_{ \mathbb{S}^{N-1}  } \,  d \varphi  \, \, \prod_{i=1}^N \, \, (\varphi^{i})^{2 n^i   }
\, \, = 
\frac{1}{2^{ \sum_{i\in [N]}  n^i }} \,  \frac{ \Gamma(\frac{N}{2})  }{ 
  \Gamma(\sum_{i\in [N]}  n^i + \frac{N}{2})  }  \,   \prod_{i=1}^N 
 \Big  (  (2  n^i - 1 )!!   \Big ),
 \end{align}
 which appears in (\ref{eq:spinintegral}), and has been proved in \cite[Appendix A]{Chayes}. We see from the definition (\ref{eq:ignoresites}) that,
\begin{equation}\label{eq:monotonXi}
\mathcal{X}( n^1, \ldots,  n^{i-1}, n^i +1,  n^{i+1},\ldots , n^N) \leq 
\mathcal{X}( n^1, \ldots,  n^{i-1}, n^i,  n^{i+1},  \ldots , n^N),
\end{equation}
for any $( n^1, \ldots,  n^N)$ and $i \in[N]$. 
 Moreover, we define for any $k \in \mathbb{N}$,
 $
 \mathcal{X}_{sup}(k)  : = \sup \big\{   \mathcal{X}( n^1, \ldots, n^N) \, \, : \, \, 
\sum_{i=1}^N n^i =  k \big\},
 $
which satisfies
\begin{equation}\label{eq:boundgoestozero}
\lim\limits_{   k \rightarrow \infty        }
\mathcal{X}_{sup}(
k) = 0.
\end{equation}
For any $e \in E$, $w \in \mathcal{W}_G$ and $i \in [N]$, 
we let $m^i_e(w)$ be the number of $i$-links on the edge $e$
and we define the vector $m^i(w) = \big (m^i_e(w) \big )_{e \in E}$. Moreover, we define $\tilde \Sigma_G \subset   \mathbb{N}^{E \times [N] }$
as the set of elements $\boldsymbol{m}\in  \mathbb{N}^{E \times [N] }$ such that, for any
$x \in \Vcal$, and any $i \in [N]$, $\sum_{y \sim x} 
m^i_{ \{x,y\} } \in 2 \mathbb{N}$,
and such that for any $e \in E \setminus \Ecal$, and any $i \in [N-1]$, $m_e^i = 0$. 
In the next calculation we use the fact that, for any function 
$f : \mathbb{N}^{E \times [N]  }\mapsto \mathbb{R}$,
\begin{multline}
\sum_{ \substack{w = (m, c, \pi)  \in \mathcal{W}^\prime_G }}
\bigg(\prod_{e\in \Ecal}\frac{\beta^{m_e}}{m_e!}\bigg) 
\bigg(\prod_{x \in \Vcal}\frac{h^{m_{\{x,g\}}}}{m_{\{x,g\}}!}\bigg) 
\bigg ( \prod_{x \in V} U_x(w) \Big ) f\big ( m^1(w),  \ldots, m^N(w)\big )
\\ = 
\, 
\sum\limits_{ \boldsymbol m \in \tilde \Sigma_G} \, \, 
\Big ( \prod_{i=1}^{N} \prod_{e\in \Ecal}
\frac{\beta^{m^i_e}}{m^i_e!} \,  \Big ) 
\Big ( \, \prod_{x \in \Vcal}
\frac{h^{m^N_{\{x,g\}  }}}{m^N_{ \{x,g\}}!}  \, 
\Big ) 
\Big ( \prod_{x \in \Vcal} \Xcal \big (\boldsymbol{n}_x(\boldsymbol{m} ) \big  )
\Big ) \, 
f\big ( m^1,  \ldots, m^N \big ),
\end{multline}
where we used the notation $\boldsymbol{n}_x = (n^1_x, \ldots, n^N_x)$ and $\bm=(m^1,\dots,m^N)$, and we wrote 
$\bn_x(\bm)=(n^1_x(m^1),\dots,n^N_x(m^N))$, where $n^i_x(m^i) : = \frac{1}{2} \, \sum_{y \sim x} m^i_{ \{x,y \} }$. For any set  $A \subset \mathcal{V}$ we now define the operator $M_A : \N^{E\times[N]} \mapsto \N^{E\times [N]}$ as follows,
$$
\forall e \in E \quad 
\forall i \in [N] \quad 
\forall \boldsymbol{m} = (m^1, \ldots, m^N) \in  \N^{E\times[N]} \quad 
\big (
M_A(\boldsymbol{m}) 
\big )^i_e: = 
\begin{cases}
m^i_e \hspace{0.2cm}  mod \hspace{0.2cm} 2  &\mbox{ if $e \in E_A$}, \\
m^i_e   & \mbox{ if  $e \in E \setminus E_A$}.
\end{cases}
$$
We obtain that, for any $\xi \in \Sigma_G$, 
\begin{multline}\label{eq:step1localtimes}
\tilde  \mu
\Big (    \big\{ \forall x \in A, n_x \geq k\big\}\cap\big\{ \forall i \in[N],  \forall e \in E, 
m^i_e \in 2 \mathbb{N} + \xi^i_e   \big\}  \Big )  
\\ 
=
\sum\limits_{ \substack{  \boldsymbol m \in \tilde \Sigma_G : \\ M_\Vcal(\boldsymbol m) = \xi }} \, \, 
\Big ( \prod_{i \in [N]} \prod_{e\in \Ecal}
\frac{\beta^{m^i_e}}{m^i_e!} \,  \Big ) 
\Big ( \, \prod_{ x \in \Vcal}
\frac{h^{m^N_{ \{x,g\}  }}}{m^N_{ \{x,g\}  }!}  \, 
\Big ) \, 
\bigg ( \prod_{x \in \Vcal} \mathcal{X}(\boldsymbol{n_x}(\bm)) \mathbbm{1}_{\big \{ \mbox{$\forall\,x \in A$,  $\sum\limits_{\substack{i \in [N] \\ y \sim x}}  m^i_{\{x,y\}   } > 2  k$}  \big \}} ( \boldsymbol{m} )\bigg ),
\end{multline}
where $\mathbbm{1}_{\mathcal{A}}(\boldsymbol{m})$ equals one if
$m \in \mathcal{A}$ and $0$ otherwise.
We now define the constant
$
\Kcal : =  \inf \big  \{ \mathcal{X}( n^1, \ldots, n^N) \, :\, 
\sum_{i=1}^N n^i \leq(d^*_\Gcal+1)/2\   \big \},
$
which is finite and positive and corresponds to the smallest ``vertex factor" (on-site weight together with the number all  possible pairings) when all incident edges (including the edge to the ghost vertex) have at most one link and, using also (\ref{eq:monotonXi}),
 we observe that  for any
$\boldsymbol{m}$ satisfying 
$\sum_{{i \in [N],  y \sim x}}  m^i_{\{x,y\}   } > 2  k$ for any $x \in A$, we have that,
$$
\prod_{x \in \Vcal} \mathcal{X} \big (\boldsymbol{n_x}(\bm) \big ) \leq 
\Big (  \frac{\mathcal{X}_{sup}(k) }{\mathcal{K}} \Big )^{|A|} \, \, 
\prod_{x \in \Vcal} \mathcal{X} \big (
\boldsymbol{n_x}\big (M_{A}(\bm) \big ).
$$
Moreover, for each edge which is incident to at least one vertex in $A$ and for any colour we use the bound,
$$
\sum_{ n \in 2 \mathbb{N}+ q  } \frac{u^n}{n!} \leq u^q \, e^u,
$$
which holds for any  $q\in \{0,1\}$.
Using such two bounds in (\ref{eq:step1localtimes}) 
and the fact that, if $\boldsymbol{m}$ 
satisfies
$\mathcal{M}_\Vcal(m) = \xi$ and  
$m_e^i = \xi_e^i$ for any $e \in E_A$ and $i \in [N]$, 
then 
$\mathcal{X} \big (\boldsymbol{n_x}
(M_A(\bm) \big ) = 
\mathcal{X} \big (\boldsymbol{n_x}
(\bm) \big )$,
we obtain that,
\begin{align}
\begin{aligned}
\tilde  \mu
\Big (    &  \big\{ \forall x \in A, n_x \geq k\big\}\cap\big\{ \forall i \in[N],  \forall e \in E, 
m^i_e \in 2 \mathbb{N} + \xi^i_e   \big\}  \Big )    \\
 \leq  
& \Big (  \mathcal{X}_{sup}(k)  \, \, \frac{ e^{ h \,  + \,  N \, \beta \, (d^*_\Gcal  +1)  }}
{\mathcal{K}} \, \, \Big )^{|A|} \, 
\sum\limits_{ \substack{  \boldsymbol m \in \tilde \Sigma_G : \\ M_\Vcal(\boldsymbol m) = \xi, \\
m_e^i = \xi_e^i\,  \forall e \in E_A, \, \forall i \in [N] }} \, 
\bigg ( \prod_{i \in[N]}
\prod_{e \in \Ecal } 
 \frac{   \beta^{m^i_e}}{m^i_e!} 
\bigg )
\, 
\bigg ( \prod_{x \in \Vcal } 
 \frac{h^{m^N_{ \{x,g\}  }}}{m^N_{ \{x,g\}  }!} \, \bigg ) 
\Big ( \prod_{x \in \Vcal} 
\mathcal{X} \big (\boldsymbol{n_x}
(\bm) \big )\Big ) 
\\
&=
\bigg (   \mathcal{X}_{sup}(k) \, \,  \frac{e^{ h \,  + N  \,  \beta \, (d^*_\Gcal  +1)    }}{\mathcal{K}}\Big )^{|A|} \, \, 
\tilde  \mu 
\Big (  \bigcap_{i \in [N]} \,  \big\{  \forall e \in E_A, m^i_e = \xi^i_e, 
  \forall e \in E \setminus E_A, 
m^i_e \in 2 \mathbb{N} + \xi^i_e \big\}  \bigg ).
\end{aligned}
\end{align}
From the previous inequality and (\ref{eq:splittingmultline}) we deduce that,
\begin{multline*}
\tilde \mu
\Big (     \forall x \in A, n_x \geq k    \Big )  \\ \leq 
\Big (  \mathcal{X}_{sup}(k) \frac{ e^{ h \,  + \, N \,  \beta \, 
(d^*_\Gcal  +1) }}{\Kcal}\Big )^{|A|} 
\sum\limits_{  \xi \in \Sigma_G }
\tilde  \mu \Big (     \forall i \in [N],  \forall e \in E \setminus E_A, 
m^i_e \in 2 \mathbb{N}+  \xi^i_e     \Big )  \\
\leq   \, \Big ( \mathcal{X}_{sup}(k)   \frac{e^{  h \,  + \,  N \, \beta \, (d^*_\Gcal  +1)   }}{\Kcal} \, \, \Big )^{|A|} \, \, \tilde \mu( \mathcal{S}_G ).
\end{multline*}
From (\ref{eq:boundgoestozero}), we deduce that the quantity inside the brackets in the right-hand side goes to zero with $k$, thus  we conclude the proof of (\ref{eq:Anbound}) and we obtain an explicit expression for the constant \hyperlink{c1}{$c_1$},
\begin{equation}\label{eq:C1definition}
c_1: =  \mathcal{X}_{sup}(k)  \, \frac{e^{  h \,  + \,  N \, \beta \, (d^*_\Gcal  +1)   }}{\Kcal}.
\end{equation}
The proof of (\ref{eq:Bnbound}) is analogous, suppose that $z \in A$, 
denote by  $\tilde  E$  the expectation with respect to $\tilde \mu$. In  the first step we have that, for any $\xi \in \Sigma_G$, 
\begin{multline*}
\tilde E\Big ( m_{ \{z,g\}} \,  \mathbbm{1}\{   \forall x \in A, n_x \geq k, \forall i \in[N], \forall e \in E, m^i_e \in 2 \mathbb{N} +   \xi_e^i   \}  \Big   )
\\
= 
\sum\limits_{ \substack{  \boldsymbol m \in \tilde \Sigma_G : \\ M_\Vcal(\boldsymbol m) = \xi }} \, \, 
\Big ( \prod_{i=1}^{N} \prod_{e\in \Ecal}
\frac{\beta^{m^i_e}}{m^i_e!} \,  \Big ) 
\bigg ( 
\prod_{x \in \Vcal \setminus \{z\}}  
\frac{ h^{m^N_{ \{x,g\}  }}}{m^N_{ \{x,g\}  }!} \, \bigg )  
\bigg ( 
\frac{ m^N_{ \{z,g\}  }  h^{m^N_{ \{z,g\}  }}}{m^N_{ \{z,g\}  }!} \, \bigg ) \\
\bigg ( \prod_{x \in \Vcal} \mathcal{X}(\boldsymbol{n_x}(\bm)) \mathbbm{1}_{\big \{ \mbox{$\forall\,x \in A$,  $\sum\limits_{i \in [N], y \sim x} m^i_{\{x,y\}   } > 2 k$}  \big  \}}(\boldsymbol{m}) \bigg ),
\end{multline*}
and in the next steps we proceed analogously to the previous case with the exception that we bound the sum associated to the edge $\{z, g\}$ by 
$
\sum_{m \in 2\mathbb{N}+q}
m \, \tfrac{h^m}{m!} \leq h^{q+1} e^h$, for $q \in \{0,1\}$. This concludes the proof.
\end{proof}

For any $\epsilon \in (0, 1)$, $x, y \in \Vcal$, and $k \in \mathbb{N}$, let   $\mathcal{E}_{x,y, \epsilon, k} \subset \mathcal{W}_G$
 be the set of configurations $w$ such  that there exists a self-avoiding nearest-neighbour path 
 in $G$, $\gamma = (x_0, x_1, \ldots, x_{\ell})$, with $x_0 = x$ and $x_{\ell}  =y$, such that at least $\epsilon \,  \ell$ vertices $z \in \gamma$ are such that $n_z(w) > k$.
Depending on the context, we might write  $\mathcal{E}_{x,y, \epsilon, k}  \subset \mathcal{M}_G$
for the set of link cardinalities $m \in  \mathcal{M}_G$ satisfying the same property.

\hypertarget{c2}{
\begin{lem}\label {lem:linkcardinalityevent}
Fix $N \in \mathbb{N}_{>0}$, $\beta, h  \geq 0$. For any $\epsilon \in (0,1)$,
there exist $C_2  = C_2 ( d_\Gcal^*, N, \beta, \epsilon, h ) \in (0, \infty)$ and  $K=   K(d_\Gcal^*, N, \beta, h) \in \mathbb{N}$ such that, for any $k \geq K$, 
$$
\forall x, y \in \mathcal{V} \quad \quad \mathbb{E}_{G, N, \beta, h} \big ( \,  m_{\{x,g\}} \, \mathbbm{1}\{  \mathcal{E}_{x,y, \epsilon, k}   \}  \, \big ) \leq C_2 \,  \, \,    e^{ \, - d_{\Gcal}(x,y) }.
$$
Moreover,  $K(d_\Gcal^*, N, \beta, \epsilon, h)$ can be chosen to be non-decreasing with  $h$.
\end{lem}
}
\begin{proof}
Given a self-avoiding walk  $\gamma$ in $\Gcal$,
we denote by $|\gamma|$ the number of vertices contained in  $\gamma$ and we write $\sum_{\gamma : x \rightarrow y}$ for  the sum over all self-avoiding walks starting from $x$ and ending at $y$. Below we apply Lemma \ref{lemma:linkcardinalityset},
 we bound the number of self-avoiding walks of length $n$ by ${d_\Gcal^*}^n$, and $\binom{n}{r} \leq 2^n$, obtaining,
\begin{align}
\begin{split}
&  \mathbb{E}_{G, N, \beta, h} \big (   m_{\{x,g\}} \, \mathbbm{1}\{  \mathcal{E}_{x,y, \epsilon, k}   \}   \big ) 
 \leq  \sum\limits_{  \gamma : x \rightarrow y    }
\sum\limits_{  \substack{ A \subset \gamma  : \\  | A | > \epsilon |\gamma|  }  }
 \mathbb{E}_{G, N, \beta, h} \big (   m_{\{x,g\}} \, \mathbbm{1}\{ 
 \forall x \in A, n_x > k   \}   \big ) 
 \leq 
\sum_{\gamma:x\to y}\sum_{r=\lceil \epsilon|\gamma|\rceil}^{|\gamma|}\binom{|\gamma|}{r} \, h \, c_1^r \\
 & \leq h\sum_{n\geq d_{\Gcal}(x,y)} {(d^*_\Gcal  +1) }^n \sum_{r=\lceil \epsilon n\rceil}^{n}\binom{n}{r}c_1^r \leq \frac{h}{1 - c_1}  \, 
 \sum_{n\geq d_{\Gcal}(x,y)}( 2\,  d_\Gcal^*\, c_1^{  \epsilon} )^{  n } \leq 
 \frac{h}{1 - c_1} \, \frac{1}{1 - 2 \, d_\Gcal^*  \, c_1^\epsilon  } \, \, \big ( 2 \, d_\Gcal^*  \, c_1^\epsilon \big )^{ d_\Gcal(x,y)  },
 \end{split}
\end{align}
where for the last inequality we assumed that $k$ is large enough so that $2 \, d_\Gcal^* \, c_1^\epsilon < 1$ (see equation (\ref{eq:C1definition})).
Choosing $k$ so large that $2 \, d_\Gcal^* \, c_1^\epsilon < e^{-1}$ gives the bound. The monotonicity property of $K$ follows from the definition of \hyperlink{c1}{$c_1$} (see equation (\ref{eq:C1definition})).
\end{proof}

\section{Sampling procedure and number of surviving walks}
\label{sect:samplingprocedure}
The main goal of this section is the proof of the following proposition. The proposition states that, conditional on having `many vertices' with `low' local time between the vertices $x$ and $y$, the expected number 
of $(\{x,g\}, \{y, g\})$-extremal walks is exponentially small in the distance between $x$ and $y$.

Below we condition on sets of configurations in $\mathcal{W}^\prime_G$ which have a prescribed link cardinality
and colouring on the edges of $\Gcal \subset G$, thus we need to ensure that the event on which we condition
has non-zero probability. For this, we introduce the notion of admissible pairs.
We say that the pair $(m,c)$, with $m \in \mathcal{M}_G$ and $c \in \mathcal{C}_G(m)$
is  \textit{admissible} if the total number of $i$-links touching any original vertex is even for any colour $i \in [N]$ and no $i$-link is on a ghost edge if $i \neq N$.
Recall the definition of the event $\mathcal{E}_{x,y,\epsilon, k}$, which was provided above Lemma \ref{lem:linkcardinalityevent}. We use the superscript $^c$ to denote the complementary event. 
In this whole section we again assume that $\beta, h \in \mathbb{R}_{>0}$.

\hypertarget{c3}{
\begin{prop}\label{prop:mainprop}
Assume that $\mathcal{G} = (\mathcal{V}, \mathcal{E})$ is a finite simple graph and that $G = (V,E)$ is obtained from $\mathcal{G}$ by adding a ghost vertex, as defined above Definition \ref{def:measure}.
Let $h, \beta \in \mathbb{R}_{>0}$, $\epsilon \in (0, 1)$ be arbitrary.
There exists a large enough $K_0 \in \mathbb{N}$ and
constants  $c_3, C_3 \in (0, \infty)$ such that, for any $k \geq K_0$,
$x, y \in \mathcal{V}$,  any admissible pair $(\tilde m, \tilde c)$ such that  $\tilde m \in \mathcal{E}^c_{x,y, \epsilon, k}$ and  $\tilde c \in \mathcal{C}(\tilde m)$,
\begin{equation}\label{eq:remainder}
 \mathbb{E}_{G, N, \beta, h} \big (    \, M_{x,y}  \, \big | \, m_e(w) = \tilde m_e, \,  c_e(w) = \tilde c_e \, \, \forall e \in \Ecal \, \big ) \,  \leq \,  C_3 \, e^{  - \, \, \frac{\epsilon}{k} \, \, c_3 \, \, d_{\Gcal}(x,y) },
\end{equation}
where $c_3 = c_3 ( d_\Gcal^*, N, \beta, h) = 
O(h^2)$ in the limit as $h \rightarrow 0$ and $K_0$ is non-decreasing with $h$.
\end{prop}
}
The idea of the proof is that every time that an N-walk `starting from $x$' (i.e, with extremal link on $\{x,g\}$) encounters a vertex with `low' local time and with at least one link on the ghost edge incident to it, with `considerable' probability it pairs to that link and  `dies' there. 
Thus, in order for the walk starting from $x$ to reach $y$, it must happen that the walk \textit{does not die} at any vertex with a link on the ghost edge incident to it. We will show that this happens with exponentially small probability in $d_{\Gcal}(x,y)$. 
The central technical tool for the proof of the proposition is a sampling procedure, which allows us to `reveal' the  number of ghost links incident to any original vertex step-by-step, thus allowing
the comparison with a simpler stochastic process.
 A similar strategy has been used also in \cite{B-T}.

\subsection{Sampling procedure}\label{sec:samplingproc}
 Recall the notation for sets of vertices and their boundaries that was introduced in Section \ref{sec:linkbounds}. 
For any $A \subset \mathcal{V}$, we define the set,
\begin{multline}
\mathcal{S}^{A}_{G} : = \Big \{ \, \,  w = (m, c, \pi) \, \, : \, \,  m \in  \mathbb{N}^{E_A}, \, \,
c \in \mathcal{C}_{E_{A}}(m),    \, \,  \pi \in \mathcal{P}_{A}(m, c), \, \,\\   u^i_x(w) = 0 \, \,   \forall x \in A,  \forall i \in [N], \, \, n^j_g  = 0 \, \,  \forall j  \neq N  \, \Big \},
\end{multline}
where $\mathcal{C}_{E_A}(m)$ is the set of colourings $c = (c_e)_{e \in E_A}$ for $m$,  and $\mathcal{P}_{A}(m, c)$ is the set of pairing functions for $(m,c)$, $\pi = (\pi_x)_{x \in A   \,  \cup \, \partial^e A \, \cup \,  \{g\}}$, such that $\pi_x = \pi_g =  \emptyset$ for any $x \in \partial^e A$ (in other words,  any link touching a vertex $x \in \partial^e A$ or $g$ is unpaired at that vertex).
We obtain that if $A = \Vcal$ (in which case  $\partial^e A = \emptyset$) then  $\mathcal{S}_G^{A}  = \mathcal{S}_G$.
Moreover,  let  $A,  B \subset \mathcal{V}$ be such that 
$\partial ^e A \subset B$ and,
 for any $\tilde w = (\tilde m, \tilde c, \tilde \pi) \in \mathcal{S}^B_{G}$, we define
$$
\mathcal{S}^{A, \tilde w}_{G} : = \Big \{ \, \,    w = (m, c, \pi) \in \mathcal{S}^A_{G},
\, \, : \, \, m_e = \tilde m_e \, \, \, \forall e \in \partial E_A
 \, \,  \Big \},
$$
to be the set of configurations in $\mathcal{S}^{A}_{G}$ which a  agree with $\tilde w$ on  $ \partial E_A$.
In other words, any   realisation in $\mathcal{S}_G^A$ consists  of loops of any colour which are entirely contained in $A$, walks of colour $i \in [N-1]$   with extremal links on the original edges in $\partial E_A$ ($\tilde w$-allowing), and  walks of colour $N$  with extremal links on the original  ($\tilde w$-allowing) or ghost edges in $E_A$.
On this set we define the probability measure,
\begin{equation}\label{eq:measuresampling}
w = (m, c, \pi     ) \in \mathcal{S}^{A, \tilde w}_{G},  \, \, 
 \mathbb{P}^{A, \tilde w}_{G,N, \beta, h}(w) : = \frac{1}{Z_{G,N, \beta, h}^{A, \tilde w}} 
 \Big ( \prod_{e \in E_A  \setminus E_A^g } \frac{\beta^{ m_e  }}{m_e!} \Big ) 
  \Big (   \prod_{x \in A   } \frac{h^{ m_{\{x,g\}}  }}{m_{\{x,g\}}!} \Big ) 
 \Big (    \prod_{x \in A} U_x(w) \Big ),
\end{equation}
where $U_x(w)$ is defined in Definition \ref{def:measure},
and $Z_{G,N, \beta, h}^{A, \tilde w}$ is a normalisation constant.
Sometimes we will omit some of the sub-scripts  for a lighter notation. 
The measure (\ref{eq:measuresampling}) can be viewed as a restriction of $\mathbb{P}_{G, N, \beta, h}$
to subsets of $G$, with a boundary condition possibly allowing walks of any colour entering and leaving $A$
from its boundary.

\textbf{Restrictions and compositions.}
Given two sets, $A, B$, such that $A \subset B   \subset \mathcal{V}$, and a realisation $w = (m, c, \pi) \in \mathcal{S}_G^{B}$,  we let $w{|}_{ A  }$  be the \textit{restriction of $w$ to the vertices of $A$}, namely $ w{|}_{A } = (m^\prime, c^\prime, \pi^\prime)$ is an element of  $\mathcal{S}_G^{A,  w}$ such that 
$$
m^\prime_e = m_e, \quad  c^\prime_e = c_e \quad \quad \forall e \in E_A,$$ and, moreover, 
$$
 \pi^\prime_x = \pi_x\quad   \forall x \in A\,   ; \quad \quad 
 \pi^\prime_x = \pi_g = \emptyset \quad   \forall x \in \partial^e A.
$$
Furthermore, given two disjoint sets of vertices, $A, B \subset \mathcal{V}$, we say that $w = (m, c, \pi) \in \mathcal{S}^A_G$ and $w^\prime = (m^\prime, c^\prime, \pi^\prime) \in \mathcal{S}^B_G$
are \textit{compatible} if they agree on $\partial E_A \cap  \partial E_B$, namely for any $e \in \partial E_A \cap  \partial E_B$ we have that $m_e = m_e^\prime$ and $c_e = c_e^\prime$  (this condition is always fulfilled when
 $\partial E_A \cap  \partial E_B = \emptyset$).
Finally, given two disjoint subsets $A, B \subset \mathcal{V}$,
and two compatible configurations, $w = (m, c, \pi) \in \mathcal{S}^A_G$ and $w^\prime = (m^\prime, c^\prime, \pi^\prime) \in \mathcal{S}^B_G$, 
we define their composition,
$$
w  \sqcup w^\prime,
$$
as the configuration in $\mathcal{S}^{A \cup B}_{G}$,
$w  \sqcup w^\prime =   (m^{\prime \prime}, c^{\prime \prime}, \pi^{\prime \prime})$ satisfying,
$$
m_e^{\prime \prime} = m_e, \, \, c_e^{\prime \prime} = c_e \, \, \, \, \, \, \forall e \in E_A,
\quad \quad 
m_e^{\prime \prime} = m^\prime_e, \, \, c_e^{\prime \prime} = c_e^{\prime} \, \, \, \, \, \, \forall e \in E_B, 
$$
$$
\pi^{ \prime \prime  }_x=  \pi_x\, \, \, \, \, \, \forall x 
\in A, \quad  \quad 
\pi^{ \prime \prime  }_x=  \pi_x^{\prime}\, \, \, \, \, \, \forall x 
\in B.
$$

The {{sampling procedure}} depends on a realisation of link cardinalities on original edges, $\tilde m \in   \mathcal{M}_{\Gcal}$,  on a colouring of such links, $\tilde c \in \mathcal{C}_{\Gcal}(\tilde m)$, and on a sequence of maps which we call 
a \emph{strategy},
$F =  ( F_{A}  )_{  A \subset \mathcal{V}},$ where,
\begin{equation}\label{eq:samplingstrategy}
\forall A \subset \mathcal{V} \quad \quad  F_{A} :  \mathcal{S}^{A}_{G} \mapsto  \mathcal{V} \setminus A.
\end{equation}
The strategy establishes which (original) vertex will be sampled next  depending on the outcome of the previous steps. 
We write $w \sim P$ to denote that $w$ is sampled according to $P$,
where  here $P$ is some unspecified probability measure.
\begin{defn}[\textbf{Sampling procedure}]\label{def:sampling}
The sampling procedure with strategy $F$, and admissible pair $(\tilde m, \tilde c)$, with 
$\tilde m \in \mathcal{M}_\Gcal$, $\tilde c \in \mathcal{C}_\Gcal(\tilde m)$,
 is defined recursively by  the following steps.
\textit{At the {first step}, $n =0$, we select the vertex $x_0 : = F_{\emptyset}( \emptyset) \in \Vcal$ and
we sample a configuration,
$$w_0^\prime \, \, \sim \, \, \mathbb{P}_{G, N, \beta, h} \Big (  \, \,   \cdot   \, \, \Big  | \, \, m_e(w) = \tilde m_e, \, c_e(w) = \tilde c_e  \, \forall e \in \mathcal{E}   \Big ),$$
which is an element of $\mathcal{S}_G = \Scal_G^{\Vcal}$.
 We define the restriction 
$w_0 : = (w^\prime_0){\big|}_{ \{x_0\}  }$, and set $A_0 : = \{x_0\}$.}
At {any step} $n > 0$, we define
$$x_n : = F_{A_{n-1}}( w_{n-1}) \in \Vcal \setminus A_{n-1},$$
(i.e., $x_n$ is chosen according to the strategy and depending on the outcome of the previous steps), which we call the `vertex  selected at the step $n$',
or, in short, `$n$-vertex'. Furthermore,
we sample the configuration,
$$
w^\prime_n \, \, \sim \, \,  \mathbb{P}^{\Vcal \setminus {A}_{n-1}, w_{n-1}}_{G, N, \beta, h}
\Big (   \, \,   \cdot   \, \, \Big  | \, \, m_e(w) = \tilde m_e, \, c_e(w) = \tilde c_e  \, \forall e \in \mathcal{E}  \cap E_{\Vcal \setminus A_{n-1}} \Big ),
$$
-- which we refer to 
 as a `$n$-sampling configuration' --
which is an element of $\mathcal{S}^{\Vcal \setminus A_{n-1}, w_{n-1}}_G$.
We define the new configuration,
$$
w_n : = w_{n-1}  \,\, \,  \sqcup  \,  \, \,\Big (  w^\prime_n \, \big |_{\{x_n\}} \Big ),
$$
(noting that the two composed configuration are compatible by construction),
which we refer to as a `$n$-composed configuration',
and the set,
$$
A_n : = A_{n-1} \cup \{ x_n  \},
$$
which we refer to as a `$n$-explored set'.
This concludes the definition of the step $n$.
We denote by $T : = | \mathcal{V} | - 1$ the last step of the procedure, which, by construction, is such that  $T = \inf \{ n \in \mathbb{N} \, \, : \, \,    A_ n = \mathcal{V}\}$.
\end{defn}
In other words, the sampling procedure defines a (random) sequence of sets $A_0 \subset A_1 \ldots \subset A_T = \mathcal{V}$ 
such that $A_n$ is obtained by adding to $A_{n-1}$ the $n$-vertex, $x_n$. Moreover, it defines a (random) sequence of configurations, $w_1$, $w_2$, $\ldots$, $w_T$, with $w_n \in \mathcal{S}^{ A_{n}  }_G$, such that each new configuration $w_{n}$ is obtained from the previous one,   $w_{n-1}$, by adding the (random) links on ghost edges incident to $x_n$ and by specifying all the pairings at $x_n$
(the number of links on original edges and their colourings are not random since they are fixed by the conditioning and specified by $\tilde m$ and $\tilde c$).
The next proposition states that the configuration $w_T$ obtained in the last step of the procedure, which by construction is an element in $\mathcal{W}_G$, 
 is distributed according to $\mathbb{P}_{G, N, \beta, h} \big ( \, \, \cdot  \, \bigm | \, \, \, m_e(w) = \tilde m_e, \,  c_e(w) = \tilde c_e \, \, \forall e \in \mathcal{E}   \big )$, the measure which was defined in Definition \ref{def:loopwalkmeasure}.
\begin{prop}\label{prop:samplingcorresp}
Let  $\mathcal{P}_{\tilde m, \tilde c, F}$ be the law of the sampling procedure with link cardinality $\tilde m \in \mathcal{M}_{\mathcal{G}}$, $\tilde c \in \mathcal{C}_\Gcal(\tilde m)$ and
 strategy $F$, where $(\tilde m, \tilde c)$ is admissible. We have that, for any $\tilde w \in \mathcal{S}_G$ such that, 
$m_e(\tilde w) = \tilde m_e$ and $c_e(\tilde w) = \tilde c_e$ for every $e\in\Ecal$,
 we have that,
\begin{equation}\label{eq:claimprop}
 \mathcal{P}_{\tilde m, \tilde c, F} \big (  \,  w_T = \tilde w   \, \big )  = \mathbb{P}_{G, N, \beta, h} 
 \Big ( \,  \tilde w \, \bigm | \,     m_e(w) = \tilde m_e, \, c_e(w) = \tilde c_e \,  \forall e \in \mathcal{E}
  \, \Big ).
\end{equation}
\end{prop}
\begin{proof}
Use $(\Omega, \mathcal{F}, \mathcal{P}_{\tilde m, \tilde c, F})$
to denote the probability space of the sampling procedure
with link cardinality $\tilde m \in \mathcal{M}_{\Gcal}$, 
colouring $\tilde c \in \mathcal{C}_{\Gcal}(\tilde m)$,
and strategy $F$. Recall that, by Definition \ref{def:sampling},
for any realisation $\omega  \in \Omega$  of the sampling procedure and any $n \in \{0, \ldots, T\}$,
$w_n^\prime = w^\prime_n (\omega)$ denotes the 
$n$-sampling configuration,
$w_n = w_n(\omega)$ denotes the $n$-composed configuration,
 $x_n = x_n(\omega)$ denotes the $n$-vertex, and
 $A_n = A_n(\omega)$ denotes the $n$-explored set. Define also 
$\tilde x_0 = F_\emptyset(\emptyset)$,
$\tilde A_0 : = \{\tilde x_0\}$,
and, for $n \in \{1, \ldots, T\}$, we define recursively,
$
\tilde x_n : = F_{ \tilde A_{n-1}  } (\tilde w|_{ \tilde A_{n-1} } ),
$
and  $\tilde A_{n} : = \tilde A_{n-1}  \cup \{ \tilde x_{n}  \}$.
The first observation is that, by Definition \ref{def:sampling},
for any $\omega \in \Omega$,
$$
w_T(\omega) = \tilde w \iff 
\forall n \in \{0, \ldots, T\} \quad x_n(\omega) = \tilde x_n, \quad 
A_{n}(\omega) = \tilde A_n, \quad 
  w^\prime_n(\omega)  \big |_{  \{\tilde x_n\}}  =  \tilde w \big |_{  \{\tilde x_n\}}.
$$
  From this, we deduce the first identity below,
  for which we also define $\tilde A_{-1} : = \emptyset$.
   For the second identity, we use 
 the definitions (\ref{eq:probmeasure}), (\ref{eq:measuresampling})
and the  conditional probability formula,
\begin{align*}
 & \mathcal{P}_{\tilde m, \tilde c, F} \big (  \,  w_T = \tilde w   \, \big )     \\
&  = \prod_{n=0}^{T} 
\mathbb{P}^{ \Vcal \setminus  {\tilde A_{n-1} },\tilde w }_{G, N, \beta, h} \Big (  w \in \mathcal{S}^{\Vcal \setminus 
{\tilde  A_{n-1}}}_G :    w_{ |_{\{\tilde x_n\}}} =\tilde w_{|_{\{\tilde x_n\}}} \, \Big | \, \,
    m_e(w) = \tilde m_e, c_e(w) = \tilde c_e\, \,  \forall e \in  E_{\Vcal \setminus \tilde A_{n-1} } \cap \mathcal{E}  \, \Big ) 
\\
& = \mathbb{P}_{G, N, \beta, h} 
 \Big ( \,  \tilde w \, \bigm | \,     m_e(w) = \tilde m_e, \, c_e(w) = \tilde c_e \,  \forall e \in \mathcal{E}
  \, \Big ).
\end{align*}
This concludes the proof.
\end{proof}
We now introduce \textit{$k$-candidate} and \textit{$k$-good vertices},  for arbitrary  $k \in \mathbb{N}$.
\begin{defn}\label{def:good}
We say that a vertex $x \in \mathcal{V}$ is \textit{$k$-candidate} for  $w = (m, c, \pi) \in \mathcal{W}_{\Gcal}$,
 or $m  \in\Mcal_{G}$,
 if 
 $$
 \sum_{y \in \mathcal{V} : y \sim x } m_{\{x, y\}}  \leq k.
$$
We say that the vertex  $x \in \mathcal{V}$ is $k$-good for 
 $w = (m, c, \pi) \in \mathcal{W}_{G}$,
  or $ m\in\Mcal_{G}$,
 if it is $k$-candidate and, additionally,
 $
m_{\{ x,g \}} > 0.
$
\end{defn}

In the sampling procedure, while the link cardinality on the original edges and their colours are  fixed and given by $\tilde m$ and 
$\tilde c$, the link cardinality on the ghost edges is random. The next lemma states that, if  the link cardinality on the original edges, $\tilde m$, is such that the vertex $z$ is $k$-candidate, then, conditional on $m_e(w) = \tilde m$ for each original edge $e$ and on the colourings, with  probability uniformly bounded from below by a positive constant  (which depends only on $k$ and on the model parameters), $z$  is also $k$-good. The existence of $k$-good vertices is  important for the proof of our main theorem since   the $N$-walks starting from $x$ might `die' at such vertices with uniformly positive probability, hence having `many' $k$-good vertices between $x$ and $y$ means it is unlikely for the walks `starting from' $\{x,g\}$ to reach  $y$.

\hypertarget{c4}{
\begin{lem}\label{lem:positivitygood}
Let $k \in \mathbb{N}$ be arbitrary, fix a link cardinality $\tilde m \in  \mathbb{N}^\mathcal{E}$,
a colouring $\tilde c \in  \mathcal{C}(\tilde m)$,
and a strategy $F$, suppose that $(\tilde m, \tilde c)$ is admissible. Let $(\Omega, \mathcal{F}, \mathcal{P}_{\tilde m, \tilde c, F})$ be the  probability space of the sampling procedure. For any $n \in \{0, \ldots T\}$,  let $\mathcal{F}_n$ be the $\sigma$-algebra of the first $n$-steps of the sampling procedure. Let $n \in \{0, \ldots T\}$ be arbitrary and, recalling  Definition \ref{def:sampling} (Sampling procedure), suppose that $\omega \in \Omega$ is such that   the $n$-vertex, $x_n = x_n(\omega)$, is a $k$-candidate for $\tilde m$. Then,
\begin{equation}\label{eq:probkgood}
\mathcal{P}_{\tilde m, \tilde c, F} \Big (  
\mbox{ $x_n$ is $k$-good for } w^\prime_n    \, \, \Big | \, \,  \mathcal{F}_{n-1}     \Big   )(\omega) \geq  c_4,
\end{equation}
where  $c_4 = c_4(N, \beta, h, k) >0$ whenever $h >0$ and, additionally, $c_4 = O(h^2)$ in the limit as $h \rightarrow 0$.
\end{lem}
}
\begin{proof}
Let $\omega \in \Omega$ be as in the statement of the lemma,
recall Definition \ref{def:sampling} and that  $A_{n} = A_{n}( \omega)  \subset \mathcal{V}$ represents the set of vertices which have been `explored' up to the step $n$, recall also the definition of the $n$-sampling configuration, $w^\prime_n(\omega) \in  \mathcal{S}_{G}^{A_n(\omega), w_{n-1}(\omega)}$.
Since we assume that $x_n(\omega)$ is $k$-candidate for $\tilde m$ we deduce -- after noting that   $x_n(\omega) \in \mathcal{F}_{n-1}$ -- that,
\begin{multline}
\mathcal{P}_{\tilde m, \tilde c, F} \Big (  
\mbox{ $x_n$ is $k$-good for } w^\prime_n    \,  \Big | \, \,  \mathcal{F}_n     \Big   )(\omega)     \\ 
= 
\mathbb{P}^{\Vcal \setminus A_{n-1}(\omega), w_{n-1}(\omega)}_{  G, N, \beta, h  } \Big (  m_{\{x_n(\omega),g\}}>0   \,  \Big | \, \, m_e(w) = \tilde m_e, \,
c_e(w) = \tilde c_e 
\, \, \, \forall e \in E_{\mathcal{V}\setminus A^{n-1}(\omega)} \cap \mathcal{E}    \Big   ).
\end{multline}
For $z\in\Vcal\setminus A_{n-1}(w)$ we define 
$q^i = q^i (z): = \sum_{ y \in \mathcal{V}: y \sim z} {\tilde m}^i_{\{z, y\}}$,
and $q = \sum_{i=1}^N q^i$, 
and note that $q^i$ is even by assumption for any colour $i \neq N$.
By using the conditional probability formula,  factorising 
(and noting that the total number of $i$-links incident to any original vertex is almost surely even for each colour $i$), we obtain that, for any $z \in \mathcal{V} \setminus A_{n-1}(\omega)$ which is $k$-candidate in $\tilde m$,
\begin{align*}
&\mathbb{P}^{\Vcal \setminus A_{n-1}(\omega), w_{n-1}(\omega)}_{  G, N, \beta, h  } \Big (  m_{\{z,g\}}(w)>0     \Big |   \, \, m_e(w) = \tilde m_e, \,
c_e(w) = \tilde c_e 
\, \, \, \forall e \in E_{\mathcal{V}\setminus A^{n-1}(\omega)} \cap \mathcal{E}    \Big   )  
\\ 
&  = \sum\limits_{\substack{ n > 0  : \\ q^N + n \in  2 \mathbb{N} } }
\mathbb{P}^{\Vcal \setminus A_{n-1}(\omega), w_{n-1}(\omega)}_{  G, N, \beta, h  } \Big (  m_{\{z,g\}}(w) = n   \, \Big | \, \, m_e(w) = \tilde m_e, \,
c_e(w) = \tilde c_e 
\, \, \, \forall e \in E_{\mathcal{V}\setminus A^{n-1}(\omega)} \cap \mathcal{E}    \Big   )  
\\
& = \frac{\sum\limits_{\substack{ n>  0 : \\ q^N + n \in  2 \mathbb{N} } } \, h^{n} \, \Ucal (\frac{q + n}{2}) 
\Big ( \prod_{i=1}^{N-1}  (q^i - 1)!!  \Big )\big  ( q^N + n - 1 \big )!!     
}{ \sum\limits_{\substack{ n \geq 0 : \\ q^N + n \in  2 \mathbb{N} } } \, h^{n}  \, \Ucal(\frac{q + n}{2})  \Big ( \ \prod_{i=1}^{N-1} ( q^i - 1)!!   \Big ) \big (   q^N + n - 1 \big )!! }  \\
& = \frac{\sum\limits_{\substack{ n>  0 : \\ q^N + n \in  2 \mathbb{N} } } \, h^{n} \, \Ucal (\frac{q + n}{2}) 
\big  ( q + n - 1 \big )!!     
}{ \sum\limits_{\substack{ n \geq 0 : \\ q^N + n \in  2 \mathbb{N} } } \, h^{n}  \, \Ucal(\frac{q + n}{2})  \big (   q + n - 1 \big )!! } 
 = 
\frac{1}{ 1 + c_5}, 
\end{align*}
where
$c_5$ is defined by the last identity and it is such that  $c_5 = 0$ if $q^N \in 2 \mathbb{N}+1$, in which case the lemma trivially holds. 
In the previous formula we also used the fact that $(n  - 1)!!$ corresponds to the total number of pairings of $n$ links of the same colour touching a vertex, where $n$ is even.
Otherwise, if $q^N \in 2 \mathbb{N}$, then,
\begin{align*}
c_5  & = \frac{ 
\Ucal(\frac{q}{2}) \, 
(q^N - 1)!!   }{ \sum\limits_{\substack{ n>  0 : \\ q^N + n \in  2 \mathbb{N} } } h^n \,  \Ucal(\frac{q + n}{2}) 
 \,  (q^N  - 1 + n)!!   }  \\
&  \leq  \frac{  \Ucal(\frac{q}{2}) \, \, (q^N - 1)!!   }{ h^2 \, \Ucal(\frac{q+2}{2}) \, \, (q^N - 1 + 2)!! }
 = \frac{1}{h^2} \,   \frac{q  + 2  }{q^N - 1 + 2}
\leq 
\frac{k+2}{h^2}    < \, \, \infty,
\end{align*}
where for the second-last inequality  we used the fact that $z$ is $k$-good by assumption.
Note that the left-hand side depends only on $k$ and $h$ and that $c_5 = O(1/h^2)$ as $h \rightarrow 0$. This concludes the proof.
 \end{proof}

\subsection{The walk-tracking sampling strategy}
We now define a specific sampling strategy, the \textit{walk-tracking sampling strategy.}
The walk-tracking  sampling strategy consists of selecting at every step $n$ a vertex,
$x_n$, which belongs to the external boundary of the $n$-explored set and  such that a walk with an extremal link on  $\{x_0,g\}$ leaves the set $A_{n-1}$ precisely from $x_n$, where $x_0$ is the $0$-vertex of the sampling procedure. 
The walk-tracking sampling strategy will allow a  comparison  with a simpler stochastic process in order  to bound from above (stochastically) the number of  walks with extremal link on  $\{x_0, g\}$ 
which reach the external boundary of $A_n$ (namely, which `survive' until the step $n$) as a function of $n$.
We will show that  this number decays  exponentially fast with $n$ and  the expected number of walks with extremal link on $\{x_0,g\}$ ever touching  $y$ will turn  out to be exponentially small with the graph distance between $x_0$ and $y$. 

Before introducing the definition of the walk-tracking sampling strategy, we introduce the notions of {surviving walks}, {escape vertex,} and {selected walk.}  
Given a set $A \subset \mathcal{V}$, a vertex $x \in A$, 
 a configuration $w \in \mathcal{S}_G^A$, and an integer $j \in \{1, \ldots, m^N_{ \{x,g\} }(w)\}$, we say that the \textit{$j$-th walk of $w$ from $x$ survives in $A$} if  there exists an edge $\{z,q\} \in \partial E_A$, $q \in A$, $z \in \mathcal{V} \setminus A$, such that a walk in $w$   with extremal link  $(\{x,g\},j)$ and with the other extremal link on $\{z,q\}$ exists. In other words, this walk starts from the \textit{$j$-th} link on $\{x,g\}$ and first leaves the set $A$ on the edge $\{z,q\}$.
 We call such a vertex $z$ the \textit{escape vertex} of the $j$-th surviving walk in $A$ for $w \in \mathcal{S}_G^A$.
Moreover, we define,
\begin{equation}\label{eq:conditionsampling}
s_{x, A}(w) : = \inf \{j \in \mathbb{N}_{>0} \, : \,  \mbox{ the $j$-th walk from $x$ in $w$ survives in $A$}   \},
\end{equation}
corresponding to the smallest index of all  walks 
of $w$ from $x$ which survive in $A$ and,  if $s_{x, A}(w) < \infty$, we call the 
\textit{$s_{x, A}(w)$-th}  walk of  $w \in \mathcal{S}_G^A$
from $x$ surviving in $A$
\textit{the selected walk  in $(x, A)$ for $w \in \mathcal{S}_G^A$.}
In other words, the selected walk is a surviving walk which whose extremal link on $\{x,g\}$ has minimal label.

\begin{defn}[walk-tracking  strategy from $x \in \mathcal{V}$]\label{def:trackingwalkstr}
We call  the strategy $F = (F_A)_{A \subset V}$  a \emph{walk-tracking  strategy} from  $x$ if it satisfies the following two properties: \textbf{1.} It starts from  $x$, namely
$
F_\emptyset(\emptyset)  = \{x \}.
$
\textbf{2.} For any   $A \subset \mathcal{V}$ such that $x \in A$, and for any $w \in \mathcal{S}_G^A$ such that,
$
s_{x, A} = s_{x, A}(w)  < \infty,
$
we have that
 $F_{A}(w) : = \{z\},$
 where $z$ is the  {escape vertex} of the selected walk in $(x,A)$ for
$w \in \mathcal{S}_G^A$. 
\end{defn}
In other words, at any step, the walk-tracking  strategy `selects'   the  escape vertex of the selected walk until such a walk `dies', after that it  `selects' the escape vertex of the next selected walk until this walk also `dies', and it continues this way until no selected walk exists. Note that the walk-tracking  strategy is not uniquely defined, there might be several walk-tracking  strategies from $x$. 

We now provide a formal definition of `\textit{death of the selected walk}' (or simply `death of the walk').
Fix  an arbitrary walk-tracking  strategy  $F$  which starts  from $x \in \mathcal{V}$, recall the definition $(\Omega, \mathcal{F}, \mathcal{P}_{\tilde m, \tilde c, F })$, and recall the definition of the $n$-composed configuration,
$w_n = w_n(\omega)$ which is provided in Definition \ref{def:sampling} (Sampling procedure).
 We say that the  selected walk \textit{dies at step $n$  of the procedure} if the selected walk, which by definition has an extremal link on the edge $\{x,g\}$ in $w_n$, also has an extremal link on the edge $\{x_n, g \}$ in $w_n$.
Note that the event `the selected walk dies at step $n$' is measurable in $\mathcal{F}_n$,
the $\sigma$-algebra generated by the first $n$ steps of the sampling procedure, 
 and we assume that it is empty if no selected walk exists.

The next lemma states that, when we perform a sampling procedure following  a walk-tracking strategy, if at step $n$ we select a  $k$-candidate vertex in $\tilde m$, then with probability uniformly bounded from below by a positive constant 
(which depends only on $k$ and on the model parameters)  the selected walk dies at step $n$.  
The lemma is a consequence of Lemma \ref{lem:positivitygood}.

\begin{lem}\label{lem:deathwithpp}
Choose an arbitrary integer $k \in \mathbb{N}$, an arbitrary  walk-tracking  strategy  $F$  which starts  from $ x \in \mathcal{V}$,  and an admissible pair $(\tilde m, \tilde c)$ with  $\tilde m \in \mathbb{N}^\mathcal{E}$,
$\tilde c \in \mathcal{C}_{\Gcal}(\tilde m)$,  moreover recall the definition $(\Omega, \mathcal{F}, \mathcal{P}_{\tilde m, \tilde c, F })$. Suppose that  $\omega \in \Omega$ is such that the vertex which we sample at step $n \in \mathbb{N}$, $x_n(\omega)$, is a $k$-candidate for $\tilde m$ and, additionally,
assume that a selected walk  in $(x_n(\omega), A_{n-1}(\omega))$ for  $w_{n-1}(\omega)$ exits.
Then,
$$
\mathcal{P}_{\tilde m, \tilde c, F} \big ( 
\mbox{the  selected walk dies at the step $n$}\, \Big | \, \mathcal{F}_{n-1} \big   )(\omega) \,  \geq \, \frac{c_4}{k+1},
$$
where \hyperlink{c4}{$c_4$} was defined in Lemma \ref{lem:positivitygood}.
\end{lem}
\begin{proof}
Recall that $x_n(\omega)$ is the vertex which we select at step $n$, which  -- by definition of the walk-tracking strategy --  is the escape vertex of the selected walk, and that the selected walk  in the $(n-1)$-composed configuration exists by assumption for the realisation $\omega \in \Omega$.
Hence,  the selected   walk in the $(n-1)$-composed configuration, $w_{n-1} = w_{n-1}(\omega)$, has an extremal link on $ \{x,g\} = \{x_0(\omega),g\}$, and it contains a link on an edge connecting a vertex in $A_{n-1} = A_{n-1}(\omega)$ to $x_n \in  \mathcal{V} \setminus A_{n-1}$, which we refer to as the \textit{escape link}.
We let $\mathcal{R}$ be the event that the escape link is paired at $x_n(\omega)$ to a link on the ghost edge $\{ x_n(\omega), g  \}$
(this event is defined to be empty if no link on the ghost edge exists).
We have that, for any $\omega \in \Omega$  as in the statement of the lemma,
\begin{multline}\label{eq:previousexpr}
\mathcal{P}_{\tilde m, \tilde c, F} \big (   
\mbox{the  selected walk dies at the step $n$}\, \Big | \, \mathcal{F}_{n-1}  \big   )(\omega)   
 =  \\
 \sum\limits_{  \ell = 1 }^\infty \1_{q^N(l,\tilde m)\in2\N}
 \mathbb{P}^{A_{n-1}(\omega), w_{n-1}(\omega)}_{  G, N, \beta, h  } \Big (  \mathcal{R} \cap \{  m^N_{\{x_n(\omega),g\}}  = \ell   \}   \, \Big | \, \, m_e(w) = \tilde m_e, \,
c_e(w) = \tilde c_e 
\,  \forall e \in E_{\mathcal{V}\setminus A_{n-1}(\omega)}  \cap \Ecal \Big   ),
\end{multline} 
where we used the notation   $q^N = q^N(\ell, \tilde m): = \ell + \sum_{ y \in \Vcal : y \sim x_n(\omega)  }
\tilde m^N_{\{x_n(\omega), y \}}.$
Moreover by Definition \ref{def:loopwalkmeasure} and by the fact that $x_n(w)$ is a $k$-candidate for $\tilde m$ by assumption, we obtain that, for any $\ell \in \N$,
\begin{align}
\begin{split}
\mathbb{P}^{A_{n-1}(\omega), w_{n-1}(\omega)}_{  G, N, \beta, h  } \big (  \mathcal{R} \, \, \big | \, \, m_e(w) & = \tilde m_e, \,
c_e(w) = \tilde c_e 
\,  \forall e \in E_{\mathcal{V}\setminus A_{n-1}(\omega)}  \cap \Ecal,   m^N_{\{ x_n(\omega),g\}}(w)  = \ell     \,  \big   ) \\ 
\label{eq:consideration1234} & = 
\frac{ \ell }{  q^N (\ell, \tilde m)}   \geq \frac{ 1 }{  k  +1  }. 
\end{split}
\end{align}
For the previous identity we used the fact that, by the definition of the probability measure (\ref{eq:measuresampling}), conditional on the link cardinalities and colouring on all the edges which are incident to a given vertex in $A \subset \Vcal$, the pairing function at that vertex has uniform distribution on the set of allowed pairings. 
By combining 
(\ref{eq:previousexpr})
and (\ref{eq:consideration1234})
and by using the fact that,  by Lemma \ref{lem:positivitygood},
conditional on  the vertex $x_n(\omega)$ being $k$-candidate, 
with probability at least \hyperlink{c4}{$c_4$} it is also $k$-good,
we deduce that, for any $\omega \in \Omega$ such that $x_n(\omega)$ is $k$-candidate,
\hypertarget{c6}{
$$
\mathcal{P}_{\tilde m, \tilde c, F} \big (   
\mbox{the  selected walk dies at the step $n$}\, \Big | \, \mathcal{F}_{n-1}  )(\omega) \geq c_6 : = 
\frac{c_4}{k+1},
$$
}
where the constant \hyperlink{c4}{$c_4$} was defined in Lemma  \ref{lem:positivitygood}.
This concludes the proof.
\end{proof}

\subsection{Stochastic comparison and proof of Proposition \ref{prop:mainprop}}
Consider a sampling procedure with walk-tracking strategy,
and denote its probability space by $(\Omega, \mathcal{F}, 
\mathcal{P}_{\tilde m, \tilde c})$, as introduced above.
 Recall the definition of the selected walk provided above Definition \ref{def:trackingwalkstr} and the definition of death of the walk provided above Lemma \ref{lem:deathwithpp}.
For any realisation of the sampling procedure $\omega \in \Omega$ we set $T_{0} = 0$ and we define recursively for any $j \in \mathbb{N}_{>0}$,
$$
T_j(\omega) : = \inf  \{  n >  T_{j-1}(\omega) \, : \mbox{ the selected walk in $A_n(\omega)$  dies at step $n$ and $x_n(\omega)$ is $k$-candidate} \},
$$
the step a selected walk dies for the $j$-th time when a $k$-candidate vertex is selected in the course the procedure, using the convention $\inf \{ \emptyset  \} = \infty$. Moreover,  for any $j \in \mathbb{N}_{>0}$,
 we denote by
$$
X_j (\omega) : =\big |   \{  n \in \{T_{j-1} (\omega)+1, \ldots, T_j (\omega)  \wedge T   \} \, :  x_n(\omega) \mbox{ is $k$-candidate in $\tilde m$} \} \big |,
$$
the number of times 
between two consecutive deaths of the selected walk that
a $k$-candidate vertex is sampled.
We now define a sequence of independent random variables,
$(Y_j)_{j \in \mathbb{N}_{>0}}$, with geometric distribution, $Ge(1 - c_6)$, each, where \hyperlink{c6}{$c_6$} is the constant which appears in Lemma \ref{lem:deathwithpp} and the average of $Y_j$ is $\frac{1}{c_6}$.
The next lemma states that the variables $Y_j$  stochastically bound from above the variables $X_j$. The reason is that, by Lemma \ref{lem:deathwithpp}, at every step on a $k$-candidate vertex the selected  walk dies with probability at least $c_6$
uniformly. The proof of the lemma is standard and it is presented in the appendix.

\begin{lem}\label{lem:domination}
Let $\tilde m \in \mathbb{N}^\mathcal{E}$ be a link cardinality on original edges,  let $\tilde c \in \mathcal{C}_{\Gcal}(\tilde m)$ be a colouring of $\tilde m$, assume that $(\tilde m, \tilde c)$ is admissible and let $F$ be a walk-tracking sampling strategy. 
Then, for any $\ell, r  \in \mathbb{N}$,
$$
\Pcal_{\tilde m, \tilde c, F}  \big (   
\sum\limits_{i=1}^\ell X_i  > r  \, \bigm | \mathcal{F}_0   \big ) \,  \leq  \, 
\mathcal{P}_{\tilde m, \tilde c, F} \big ( 
\sum\limits_{i=1}^\ell Y_i  > r \, \big   ),
$$
where  we use $\mathcal{P}_{\tilde m, \tilde c, F}$ also for the law of the   variables $(Y_n)_{n \in \mathbb{N}}$, which we assume to be defined in the same probability space of the sampling procedure and which are independent from the sampling procedure.
\end{lem}
We are now ready to present the proof of Proposition \ref{prop:mainprop}.
\begin{proof}[\textbf{Proof of Proposition \ref{prop:mainprop}}]
Suppose that $h, \beta > 0$.
Choose a  pair of vertices $x, y \in \mathcal{V}$,
a link cardinality 
$\tilde m \in \mathcal{E}^c_{x,y, \epsilon, k}$, and a
 colouring $\tilde c \in \mathcal{C}(\tilde m)$ such that $(\tilde m, \tilde c)$ is admissible,
 let $F$ be a walk-tracking sampling strategy from $x$.
Let $\epsilon \in (0, 1)$ and  $k \in \mathbb{N}$ be arbitrary. 
Define the random variables in $(\Omega, \mathcal{F}, \mathcal{P}_{\tilde m, \tilde c,  F})$,
\begin{align*}
\tau & : = \inf \big  \{ u \in  \{1, \ldots, m_{\{x,g\}}(w_T)\}
 \,  : \,      \sum\limits_{i=1}^{u  } X_i > \epsilon \, d_{\Gcal}(x,y)             \big   \}, \\ 
\tau^\prime & : = \inf \big  \{ u \in 
\N \,  : \,      \sum\limits_{i=1}^u Y_i > \epsilon \, d_{\Gcal}(x,y)             \big   \},
\end{align*}
using the convention that  $\inf \{  \emptyset \} = \infty$. 
To begin, note that,
\begin{align*}
\mathbb{E}_{G, N, \beta, h} \Big (   \, M_{x,y} \,  \bigm | 
  m_e(w) = \tilde m_e,  \, \, c_e(w) = \tilde c_e, \, \,  \forall e \in \mathcal{E}  \Big ) 
\nonumber &  =
{E}_{\tilde m, \tilde c,  F} \big (    M_{x,y}(w_T)  \big )  \\
 & 
 \leq 
{E}_{\tilde m, \tilde c, F} \Big (   \,  \big  ( \, m_{\{x,g\}}(w_T)  - \tau \, \big ) \mathbbm{1}\{ m_{\{x,g\}}(w_T)  >  \tau  \}  \Big ),
\end{align*} 
where for the first identity we used Proposition \ref{prop:samplingcorresp}, while for the inequality we used the fact that, by assumption on $\tilde m$, any self-avoiding path connecting $x$ to $y$ contains at least $\epsilon \,d_{\Gcal}(x,y)$ $k$-candidate vertices, hence none of the walks with extremal link $\{x,g\}$ which died before the procedure selects  at least $\epsilon \, d_{\Gcal}(x,y)$ $k$-candidate  vertices can reach $y$.

Now fix the integer $\ell  : =   [   \frac{ \epsilon \, d_{\Gcal}(x,y) \,  c_6}{4}  ]$.
We have that,
\begin{equation}\label{eq:combination0}
\begin{aligned}
&{E}_{\tilde m, \tilde c, F} \Big (     ( m_{\{x,g\}}(w_T)  - \tau ) \mathbbm{1}\{ m_{\{x,g\}}(w_T)  >  \tau  \}  \Big )  
\\ 
&\leq 
{E}_{\tilde m, \tilde c, F} \Big (     m_{\{x,g\}}(w_T)   \mathbbm{1}\{m_{\{x,g\}}(w_T) > \ell  \}  \Big )  + 
{E}_{\tilde m, \tilde c, F} \Big (     ( m_{\{x,g\}}(w_T)  - \tau ) \mathbbm{1}\{ m_{\{x,g\}}(w_T)  >  \tau, m_{\{x,g\}}(w_T) \leq  \ell  \}  \Big )
 \\ 
&\leq 
\mathbb{E}_{G, N, \beta, h}  \big (     m_{\{x,g\}}  \mathbbm{1}\{m_{\{x,g\}} > \ell  \}  \big )  + 
{E}_{\tilde m, \tilde c, F} \big ( m_{\{x,g\}}(w_T) \mathbbm{1}\{ \tau \leq  \ell  \}  \big )
 \\ 
 &\leq 
h \, c_1^{\ell}    + 
{E}_{\tilde m, \tilde c, F} \big ( m_{\{x,g\}}(w_T) \mathbbm{1}\{ \tau \leq  \ell  \}  \big ),
\end{aligned}
\end{equation}
where the last inequality follows from 
Lemma \ref{lemma:linkcardinalityset},  \hyperlink{c1}{$c_1$} $: = c_1(d^*_{\Gcal}, k, N, \beta, h)$ was defined there and goes to zero as $k$ goes to infinity.  We are now going to bound the second term in the right-hand side. For this, note that,
\begin{align}\label{eq:combination}
\begin{split}
{E}_{\tilde m, \tilde c, F} \big ( m_{\{x,g\}}(w_T) \mathbbm{1}\{ \tau \leq  \ell  \}  \big )  & = 
{E}_{\tilde m, \tilde c, F}  \Big ( {E}_{\tilde m, \tilde c, F} \big ( \, m_{\{x,g\}}(w_T) \mathbbm{1}\{ \tau \leq  \ell  \}   \bigm |      \mathcal{F}_{0} \, \big ) \Big ) \\   & = 
{E}_{\tilde m, \tilde c, F}  \Big (   m_{\{x,g\}}(w_0) \, \, {E}_{\tilde m, \tilde c, F} \big (  \mathbbm{1}\{ \tau \leq  \ell  \}   \bigm |      \mathcal{F}_{0} \, \big ) \Big )
\\   & 
\leq
{E}_{\tilde m, \tilde c, F}  \Big (   m_{\{x,g\}}(w_0) \, \, P_{\tilde m, \tilde c, F} \big ( 
\sum\limits_{i=1}^\ell X_i  > \epsilon \,  d_{\Gcal}(x,y)
 \bigm |      \mathcal{F}_{0} \, \big ) \Big ) \\
& \leq 
{E}_{\tilde m, \tilde c, F}  \big (   m_{\{x,g\}}(w_0)  \big )  \, \, 
{P}_{\tilde m, \tilde c,  F}  \big (   
\sum\limits_{i=1}^\ell Y_i  > \epsilon \,  d_{\Gcal}(x,y)
  \big )  \\
  & 
\leq C_7 \,    e^{     -  \epsilon \,  d_\Gcal(x,y)  \,   \frac{c_6}{10} }
\end{split}
\end{align}
where for the first identity we used the fact that,  since the sampling procedure starts from $x$, we have that $m_{\{x,g\}}(w_0) = m_{\{x,g\}}(w_T)$, for the second inequality we used the fact that the variables $Y_j$ are independent from the sampling procedure, for the third inequality we used the fact that ${E}_{\tilde m, \tilde c, F}  \big (   m_{\{x,g\}}(w_0)  \big ) \leq C_7$ for some finite constant $C_7 \in (0, \infty)$ by Lemma \ref{lemma:linkcardinalityset} and the Chernoff bound for sum of i.i.d. geometric random variables,
$P_{\tilde m, \tilde c,  F}  \big (   
\sum_{i \in [\ell] } Y_i  > \lambda \, \rho  \big ) \leq e^{  - c_6 \,  \rho \, (\lambda - 1 - \ln(\lambda)) },$
where $\rho = \frac{1}{4} \, \epsilon  \, d_\Gcal(x,y)$ is the average of the sum of variables and $\lambda$ is any positive real value.
By combining (\ref{eq:combination0}) and (\ref{eq:combination}) 
and recalling that $c_6 = \frac{\hyperlink{c4}{c_4}}{k+1}$, we obtain,
$$
E_{G, N, \beta, h} \big  ( M_{x,y} \, \, \big | \, \, m_e(w) = \tilde m_e, c_e(w) = \tilde c_e \forall e \in \Ecal  \big   )
\leq h e^{  - \epsilon \, d_\Gcal(x,y) \, \log(\frac{1}{c_1})  \frac{c_4}{4(k+1)}     } 
\, + \, 
C_7 \,    e^{     -  \epsilon \,  d_\Gcal(x,y)  \,   \frac{c_4}{10 \, (k+1)} }.
$$
Thus we deduce that, 
$$
E_{G, N, \beta, h} \big  ( M_{x,y} \, \, \big | \, \, m_e(w) = \tilde m_e, \, c_e(w) = \tilde c_e \, \, \forall e \in \Ecal  \big   )  \leq 
C_3 \, e^{ - \frac{\epsilon}{k} \, c_3 \, \,   d_\Gcal(x,y)   }, 
$$
for some positive constant $C_3 = C_3 ( N, \beta, h, d_\Gcal^*  )$ and, 
\begin{equation}\label{eq:c3definition}
c_3  : = 
\frac{1}{40}  \, c_4  \,  \min \{   \log(  \frac{1}{c_1} ) , 1 \} \,> 0 .
\end{equation}
We note that there exist $K_0$ large enough such that for any $k \geq K_0$, we have that 
$\log (\frac{1}{c_1}) > 1$ for any $h \in (0, 1)$ (recall equation \ref{eq:C1definition}). This implies that, under such a choice of $k$, $ \log ( \frac{1}{c_1} ) = O(1)$ in the limit as $h \rightarrow 0$. 
Thus, $c_3 = O( c_4) = O(h^2)$ in the limit as $h \rightarrow 0$
uniformly in the admissible pairs $(\tilde m, \tilde c)$ (recall that \hyperlink{c4}{$c_4$} was introduced in Lemma \ref{lem:positivitygood}). This concludes the proof.
\end{proof}

\section{Proof of Theorem \ref{thm:maintheorem} and extensions}
\label{sect:proof of theorem}
In this section we prove Theorem \ref{thm:maintheorem} and discuss its extensions.

\subsection{Proof of Theorem \ref{thm:maintheorem}}
Consider a finite simple graph $\mathcal{G}$ and define $G$ by adding a ghost vertex to $\mathcal{G}$ as described above.
We first use Proposition \ref{prop:looprep} and
Lemma \ref{lem:colourswitchlem} and after that, using the fact that $M_{x,y} \leq m_{\{x,g\}}$, we obtain that, for any $\epsilon \in (0, 1)$, $k \in \mathbb{N}$, $h > 0$, $\beta > 0$, $N \in \mathbb{N}_{>1}$,
\begin{align}
\begin{split}
\langle  \varphi^1_x \, \varphi_y^1 \rangle^{spin}_{ \mathcal{G}, N, \beta, h}
 & = 
\frac{1}{h^2} \mathbb{E}_{G, N, \beta, h} \big (  M_{x,y}  \big ) \\ & 
\label{eq:termtobound} \leq 
 \frac{1}{h^2}  \mathbb{E}_{G, N, \beta, h} \big (   \mathbbm{1}\{ \mathcal{E}_{x,y, \epsilon, k}  \}   \, m_{\{x,g\}}  \big ) \, + \, \frac{1}{h^2}   \mathbb{E}_{G, N, \beta, h} \big (  
 \mathbbm{1}\{ \mathcal{E}^c_{x,y, \epsilon, k}  \}   M_{x,y}  \big ),
 \end{split}
\end{align}
where $^c$ denotes the complement of the event, and the event 
$ \mathcal{E}_{x,y, \epsilon, k}$ was defined above Lemma 
\ref{lem:linkcardinalityevent}. 
We now fix $\epsilon = \frac{1}{10}$ and $k = \max\{ K_0(d_\Gcal, N, \beta, h), K(d_\Gcal, N, \beta, h), K_0(d_\Gcal, N, \beta, 1), K(d_\Gcal, N, \beta, 1)\}$, this allows us to use  Lemma \ref{lem:linkcardinalityevent}
and Proposition \ref{prop:mainprop}, where these constants have been introduced.
From Lemma \ref{lem:linkcardinalityevent} we deduce that,
\begin{equation}\label{eq:firstbound}
 \mathbb{E}_{G, N, \beta, h} \big (  m_{\{x,g\}} \mathbbm{1}\{ \mathcal{E}_{x,y, \epsilon, k}  \}   \big ) \leq C_2 \,  e^{- \,  d_\Gcal(x,y)   }.
\end{equation}
For the second term in the right-hand side of (\ref{eq:termtobound}) we use Proposition \ref{prop:mainprop} and obtain,
\begin{multline}\label{eq:secondbound}
  \mathbb{E}_{G, N, \beta, h} \big (  
 \mathbbm{1}\{ \mathcal{E}^c_{x,y, \epsilon, k}  \}   M_{x,y}  \big )   \\ =
  \sum\limits_{\substack{
 \tilde m \in \mathcal{M}_{\Gcal} \, :  \\  \tilde m \in   \mathcal{E}^c_{x,y, \epsilon, k} } \,  }
\sum\limits_{ \substack{ \tilde c \in \mathcal{C}_{\Gcal}(\tilde m) \\ (\tilde m, \tilde c) \mbox{ \tiny admissible  }  }}
  \mathbb{P}_{G, N, \beta, h} \big (   \forall e \in \Ecal,    m_e(w) = \tilde m_e, c_e(w) = \tilde c_e \big )  \, \\ \mathbb{E}_{G, N, \beta, h} \big (    M_{x,y}  \, \big | \, 
  \forall e \in \Ecal,    m_e(w) = \tilde m_e, c_e(w) = \tilde c_e  \, \big ) \leq C_3 \, e^{  - \frac{1}{10} \frac{1}{k} \, c_3  \, d_{ \Gcal}(x,y)  }.
\end{multline}
Combining the previous expression with   (\ref{eq:firstbound}) in (\ref{eq:termtobound}) we obtain (\ref{eq:mainclaim}).
Note that the monotonicity properties of $K_0$ and $K$ guarantee that the chosen value of $k$  does not depend on $h$ for $h \in (0,1)$. Thus we deduce that the exponent in the right-hand side of the inequality in (\ref{eq:secondbound}) is $O(h^2)$ in the limit as $h \rightarrow 0$. This implies that \hyperlink{c0}{$c_0$} $ = O(h^2)$.
Now let $\Gcal$ be an infinite simple graph of bounded degree, let $(\Gcal_L)_{L \in \mathbb{N}}$ be a sequence of finite simple graphs such that $\Gcal_L \subset \Gcal$. 
By noting that $d_{\Gcal_L}(x,y) \geq d_{\Gcal}(x,y)$ for any $L \in \mathbb{N}$ and that the constants $c_3$ and $C_3$ do not depend on $L$, the proof of the theorem is concluded.

\endproof

\subsection{Extensions}
\label{sect:extensions}
A first natural extension of our main result is to the Spin O(N) model in $\mathbb{Z}^d$ in the presence of non-homogeneous coupling constants and a non-zero external magnetic field. More precisely, let $J = (J_{x,y})_{x, y \in \mathbb{Z}^d}$
be a matrix of non-negative real numbers
such that for any $x, y \in \mathbb{Z}^d$,  $J_{x,y} = J_{y,x}$ and $J_{x,x,} = 0$,
let $(\Lambda_L)_{L \in \mathbb{N}}$ be an infinite sequence
of subsets of $\mathbb{Z}^d$,
with $\Lambda_L \subset \Lambda_{L+1} \subset  \mathbb{Z}^d$.
For any set $\Lambda \subset \mathbb{Z}^d$ define the hamiltonian function,
\begin{equation}\label{eq:generalhamiltonian}
\mathcal{H}_{\Lambda, J, h} (\varphi) : = - \sum\limits_{x, y \in \Lambda } \, \, J_{x,y}  \, \, \varphi_{x} \cdot \varphi_y -  h \, \, \sum\limits_{x \in \Lambda } \varphi^N_x.
\end{equation}
Under the assumptions that the coupling constants have finite range and are uniformly bounded, namely 
$$
\exists K < \infty \, : \, \,\forall x, y \in \mathbb{Z}^d, 
 J_{x,y} \leq K \mbox{ and }  J_{x,y} = 0 \mbox{ if } \|x - y \|_1 > K,
$$
our main theorem holds also, with few adaptations in the proof being required.
Note also that, since our theorem holds for arbitrary graphs, it is not difficult to account for finite boxes in $\mathbb{Z}^d$ with periodic or empty boundary conditions. 
Finally, our method also works for spin systems with continuous symmetry, whose associated measure is not necessarily expressed in terms of Boltzmann weights. For example,  consider  the following  expectation operator, 
$\langle\cdot\rangle^{(k)}_{\Gcal,N,\beta,h}$,
with  $k \in \mathbb{N}$,
 $\beta, h \geq 0$, defined by 
\begin{equation}\label{eq:k-truncated measure}
\langle f\rangle^{(k)}_{\Gcal,N,\beta,h} : =\frac{1}{Z^{(k)}_{\Gcal,N,\beta,h}}\int_{\Omega_{\Gcal,N}}\mathrm{d}\varphi \, \,  \Big ( \prod_{ \{x, y\} \in \mathcal{E}  } \big ( 
\sum\limits_{\ell=0}^k  \beta^\ell  \, \frac{  (\varphi_{x} \cdot \varphi_y)^\ell}{\ell!} \big )   \Big )  \, \, \exp\big(  h \sum\limits_{x \in \Vcal } \varphi^N_x \big) \, \,  \, f(\varphi),
\end{equation}
for any function $f:\Omega_{\Gcal,N}\to \R$, where
$Z^{(k)}_{\Gcal,N,\beta,h}$ is a normalising constant that ensures $\langle 1\rangle^{k}_{\Gcal,N,\beta,h}=1$,
the graph $\mathcal{G} = (\Vcal, \Ecal)$ is finite and the other terms have been introduced in Section \ref{sec:definition}.
We refer to such an expectation operator as the \textit{$k$-truncated Spin $O(N)$ model}.
Note that the measure (\ref{eq:k-truncated measure}) corresponds to the Spin $O(N)$ model when $k  = \infty$. 
The case $k = 1$ and $h = 0$, has been considered in \cite{Chayes, D-Copin}, its corresponding loop representation
(which, in $\mathbb{Z}^d$, uses colours and pairings like ours \cite{Chayes}, while in the hexagonal lattice takes a simpler form \cite{D-Copin}) 
is a model of interest known as the 
\textit{loop O(N) model}. This model is interesting, for example, for its connections to Schramm-L\"owner evolution and other planar statistical mechanics models.
Our measure (\ref{eq:k-truncated measure})
interpolates between the two models as $k$ is varied between $1$ and $\infty$.
For this model, our main result can be reformulated as follows
(recall the definitions in Section \ref{sec:definition}).
\begin{thm}\label{thm:maintheorem2} 
Let $\mathcal{G}$ be an infinite simple graph with  bounded degree.
For any $h\neq 0$, $\beta \geq 0$, $k \in \mathbb{N}_{>0}$, $N\in \N_{\geq 2}$, there are constants $c_8 = c_8( \Gcal, 
N, \beta, h, k)$, $C_8 = C_8( \Gcal, 
N, \beta, h, k)$  such that the following holds. Let  $(\mathcal{G}_L)_{L \in \mathbb{N}}$, with 
$\mathcal{G}_L = ( \mathcal{V}_L, \mathcal{E}_L)\subset \Gcal$,
be an arbitrary sequence of finite graphs. Then,
for any $L \in \mathbb{N} $, for any $x, y \in \mathcal{V}_L$,
\begin{equation}\label{eq:claimtruncated}
 0 \leq \langle S^1_x S^1_y\rangle^{(k)}_{\mathcal{G}_L,N,\beta,h} \leq C_8 \, e^{-c_8 \, d_{\Gcal}(x,y)},
\end{equation}
where $d_{\Gcal}(x,y)$ denotes the graph distance between $x$ and $y$. Moreover, the choice of $c_8$ can be made so that $c_8 = O(h^2)$ in the limit as $h \rightarrow 0$.
\end{thm}
Note that (\ref{eq:k-truncated measure}) does not necessarily make physical sense as a spin system for all values of $\beta \geq 0$,
since the measure  (given by $d \varphi$ times the interaction term) might be signed if $\beta$ is large.
Despite that, the spin-spin correlation, in the left-hand side of (\ref{eq:claimtruncated}), is non-negative 
and exhibits exponential decay for any non-zero value of the external magnetic field and for any $k \in \mathbb{N}_{>0}$, as our theorem states. 
The random path model associated to (\ref{eq:claimtruncated})
is completely analogous to the $k = \infty$ case, the only difference is that on every original edge at most $k$ links are allowed. 
Thus, all the steps of our proof apply with almost no difference 
(and the results of Section \ref{sec:linkbounds} are not necessary since every vertex is a.s. $d_{\Gcal}^*k$-candidate).
In particular, our result implies that its two point-function (defined as the ratio of partition functions with a $1$-walk connecting $x$ and $y$ and one without)  decays exponentially in the graph distance between $x$ and $y$.

\section*{Appendix}
\textbf{\textit{{Formal definition of paths, walks and loops.}}}
We will first define paths, which represent a connected set of links, after that we will introduce two classes of paths, walks (open paths) and loops (closed paths). 
Given $w \in \mathcal{W}_{\Gcal}$, we use $(\{x, y\}, p)$ to denote the $p^{th}$ link of $w$ which is on the edge $\{x,y\}$, with 
 $p \in \{1, \ldots,  m_{\{x,y\}}(w) \}$.
We say that a set of links \textit{$S$} in $w$,
$$
S = \big \{
(\{x_1, y_1\}, p_1),
(\{x_2, y_2\}, p_2),
\ldots 
((x_\ell, y_\ell), p_\ell)
\big \},$$
is \textit{pairing-connected in  $w$} if, for any 
pair of links, $(\{x, y\}, p)$,
$(\{x^\prime,y^\prime\}, p^\prime) \in S$,
there exists an ordered sequence of links in  $S$,
 $\big (
(\{x^\prime_1, y^\prime_1\}, p^\prime_1),$
$(\{x^\prime_2, y^\prime_2\}, p^\prime_2),
\ldots
(\{x^\prime_k, y^\prime_k\}, p^\prime_k)
\big ) \subset S$
such that the following two conditions hold simultaneously:
\begin{enumerate}[(i)]
\item $(\{x, y\}, p) = (\{x^\prime_1, y^\prime_1\}, p^\prime_1)$, and
$(\{x^\prime, y^\prime\}, p^\prime) = (\{x^\prime_k, y^\prime_k\}, p^\prime_k)$,
\item for any $i \in \{1, \ldots, k-1\}$,
$y^\prime_i = x^\prime_{i+1}$
and $(\{x^\prime_i, y^\prime_i\}, p^\prime_i )$
is paired to  $(\{x^\prime_{i+1}, y^\prime_{i+1}\}, p^\prime_{i+1} )$ at 
$y^\prime_i = x^\prime_{i+1}$.
\end{enumerate}
Paths are maximal pairing-connected sets. 
More precisely, a set of links $S$ of $w$ is a \textit{path} in $w$ if it is pairing-connected and there exists no pairing-connected set of links in  $w$, $S^\prime$, which is such that  $S^\prime  \supset S$  and $S^\prime \neq S$.
It is necessarily the case that all links belonging to the same path have the same colour.

We will now distinguish between different type of paths.
A path $S$ of $w$ is called a \textit{loop} if it is such that any link $(\{x,y\}, p) \in S$  is paired to another link at both its end-points. 
A path $S$ of $w$ is called a \textit{walk} if  $|S| = 1$  or if $|S|\geq 2$ and there exist precisely two distinct links in $S$ such that each of them is unpaired at  one end-point and paired at the other end-point.
Two such links will be called \textit{extremal links} for the walk
or extremal links for $w$.
From these definitions it follows that any path is either a loop or a walk, there are no other possibilities.

\textit{\textbf{Proof of Lemma \ref{lem:domination}.}}
Our goal is to show that, for any $u \in \{0, 1, \ldots \ell - 1\}$, 
\begin{equation}\label{eq:iteration}
{\Pcal}_{\tilde m, \tilde c,  F}  \big (   
\sum\limits_{i=1}^{\ell - u} X_i  + 
\sum\limits_{i=\ell - u+1}^\ell Y_i   > r   \bigm | \mathcal{F}_0 \big ) \leq 
{\Pcal}_{\tilde m, \tilde c,  F}  \big (   
\sum\limits_{i=1}^{\ell - u - 1} X_i  + 
\sum\limits_{i=\ell - u}^\ell Y_i   > r   \bigm | \mathcal{F}_0 \big ).
 \end{equation}
 Using (\ref{eq:iteration}) iteratively we deduce the lemma.
To begin, fix an arbitrary integer  $u \in \{0, 1, \ldots \ell - 1\}$, 
and observe that,
\begin{multline}\label{eq:previous666}
{\Pcal}_{\tilde m, \tilde c, , F}  \big (   
\sum\limits_{i=1}^{\ell - u} X_i  + 
\sum\limits_{i=\ell - u+1}^\ell Y_i   > r   \bigm | \mathcal{F}_0 \big ) \\
 =
{E}_{\tilde m, \tilde c, F} \Big (
{\Pcal}_{\tilde m, \tilde c,  F}  \big (   
X_{\ell - u}   > r  - 
\sum\limits_{i=1}^{\ell - u - 1} X_i   - \sum\limits_{i=\ell - u+1}^\ell Y_i   \Bigm | \mathcal{F}_{ T_{\ell - u - 1}  } ,    Y_{  \ell - u + 1}, 
  \ldots, Y_{  \ell}   \big )
\Big ),
\end{multline}
where ${E}_{\tilde m, \tilde c, F}$ denotes the expectation with 
respect to ${\Pcal}_{\tilde m, \tilde c, , F}$, 
the conditioning is on the whole history of the sampling procedure up to the stopping time $T_{\ell - u - 1}$ and on the   variables $Y_i$ with $i$ from $\ell - u + 1$ 
to $\ell$, these variables are independent of
$\mathcal{F}_{ T_{\ell - u - 1}  }$.
For  a lighter notation, use now $\tilde P( \cdot   )$
for ${\Pcal}_{\tilde m, \tilde c, F}  \big (   
\cdot \, \bigm | \mathcal{F}_{ T_{\ell - u - 1}  } ,    Y_{  \ell - u + 1}, 
  \ldots, Y_{  \ell}   \big )$, 
and $\tilde E$ for the expectation with respect to $\tilde P$.  
Additionally, we  set $t_{-1}: =  -1$ and  recursively define the variables,
$$
j \in \mathbb{N} \quad \quad t_j(\omega) : = \inf  \{  n >  t_{j-1}(\omega) \, : \mbox{ $x_n(\omega)$ is $k$-candidate} \},
$$
representing the times a $k$-candidate vertex is selected by the sampling procedure, again using the convention that $\inf \{ \emptyset \} = \infty$. For any $j \in \mathbb{N}$,  we denote by
$
d_j(\omega)
$
the step of the sampling procedure such that a selected walk dies for the $j$-th time at a $k$-candidate vertex,
i.e,  
$
t_{d_j}(\omega) :  = T_j(\omega),
$
for any integer $j \in \mathbb{N}$.
Note that,  for any $q \in \mathbb{N}$, 
\begin{equation}
\begin{aligned}
\tilde {\Pcal}  & \big (    
X_{\ell - u}   > q   \big )  
\\
&= 
\tilde   \Pcal \Big ( \{  \mbox{no selected walk dies at the steps } 
t_{d_{\ell - u-1}}, t_{d_{\ell - u-1}+1}, \ldots, t_{d_{\ell - u - 1} + q} \} \cap \{ t_{d_{\ell - u - 1} + q} < \infty   \}  \Big )
 \\
&= 
  \tilde {E} \Big ( 
   \tilde  P  \big ( 
   \{  t_{d_{\ell - u - 1} + q} < \infty \} \, \cap \, 
    \{ \mbox{the selected walk does not die at the step
    $ t_{d_{\ell - u - 1} + q}$ \} }  
    \bigm | \mathcal{F}_{{t_{d_{\ell - u-1}} + q - 1}} \big ) 
     \\
     &\qquad\qquad\mathbbm{1}  
     \{ \mbox{the selected walk does not die at the steps }  t_{d_{\ell - u-1}}, \ldots, t_{d_{\ell - u - 1} + q} \}
\cap \{  t_{d_{\ell - u - 1} + q} < \infty   \}     
      \Big ) 
      \\
   &  \leq (1 - c_6)  \, \, 
     \tilde P \big ( 
      \mbox{the selected walk does not die at the steps } t_{d_{\ell - u-1}}, t_{d_{\ell - u-1}+1}, \ldots, t_{d_{\ell - u - 1} + q-1}  \big ),
\end{aligned}
\end{equation}
where for the previous step we used Lemma \ref{lem:deathwithpp}.
Iterating the previous bound, we deduce that,
$$
{P}_{\tilde m, \tilde c, F}  \big (   X_{\ell - u}   > q   \bigm | \mathcal{F}_{ T_{\ell - u - 1}  } ,    Y_{  \ell - u + 1}, 
  \ldots, Y_{  \ell}   \big )
 \leq (1 - c_6)^q = 
{P}_{\tilde m, \tilde c, F}   (   Y_{\ell - u} > q ).
$$
Using the previous inequality in (\ref{eq:previous666}) and the fact that $q$ was arbitrary, we deduce
(\ref{eq:iteration})
and thus conclude the proof.

\end{document}